\documentclass[12pt]{amsart}

\title{Formal Degree of Regular Supercuspidals}
\author{David Schwein}
\date{3 January 2021}
\address{530 Church Street, Ann Arbor, MI 48105}

\usepackage{bbm}

\usepackage{comment}

\usepackage[margin=1in]{geometry}

\usepackage{tikz-cd}

\usepackage[final,notref,notcite]{showkeys}

\usepackage{textgreek}

\usepackage%
  [colorlinks,
  bookmarksdepth=subsubsection,
  allcolors=cyan,
  unicode]%
{hyperref}
\usepackage[capitalize]{cleveref}
  
\usepackage{amsmath,amssymb,amsthm}

  \newtheorem{conjecture}[equation]{Conjecture}
  
  \newtheorem{corollary}[equation]{Corollary}

  \newtheorem{lemma}[equation]{Lemma}
  \newtheorem{theorem}[equation]{Theorem}
  \newtheorem{theoremx}{Theorem} 

  \theoremstyle{definition}
  \newtheorem{claim}[equation]{Claim}
  \newtheorem{definition}[equation]{Definition}
  \newtheorem{remark}[equation]{Remark}
  \newtheorem{warning}[equation]{Warning}

  \renewcommand\C{\mathbb C}
  
  \renewcommand\G{\mathbb G}
  \newcommand*\dif{\mathop{}\!\mathrm{d}} 
  \newcommand\N{\mathbb N}
  \newcommand\R{\mathbb R}
  \newcommand\Z{\mathbb Z}
  \newcommand\one{\mathbbm 1}

  \newcommand\RS{R} 
  \newcommand\RSS{R'} 

  \newcommand\Ad{\tn{Ad}}
  \DeclareMathOperator\codim{codim}
  \DeclareMathOperator\cond{cond}
  \DeclareMathOperator\depth{depth}
  \newcommand\disc{\tn{disc}}
  \newcommand\Frob{\tn{Frob}}
  \DeclareMathOperator\PGL{PGL}
  \DeclareMathOperator\GL{GL}
  
  \DeclareMathOperator\Ind{Ind}
  \DeclareMathOperator\cInd{c-Ind}
  \DeclareMathOperator\Inf{Inf}
  \DeclareMathOperator\len{len}
  \DeclareMathOperator\Lie{Lie}
  \DeclareMathOperator\ord{ord}
  \DeclareMathOperator\rank{rank}
  \DeclareMathOperator\Res{Res}
  \DeclareMathOperator\SL{SL}
  \DeclareMathOperator\Spec{Spec}
  \DeclareMathOperator\vol{vol}

  \newcommand\cal{\mathcal}
  \renewcommand\frak{\mathfrak}
  \newcommand\tn{\textnormal}
  \newcommand\ul{\underline}

  \newcommand\defeq{:=}
  \newcommand\eqdef{:=}
  \newcommand\into{\hookrightarrow}
  
  \newcommand\xto{\xrightarrow}

\tikzset{
  symbol/.style={
    draw=none,
    every to/.append style={
      edge node={node [sloped, allow upside down, auto=false]{$#1$}}}
  }
}
\begin{document}

\begin{abstract}
Supercuspidal representations
are usually infinite-dimensional,
so the size of such a representation
cannot be measured by its dimension;
the formal degree is a better alternative.
Hiraga, Ichino, and Ikeda conjectured a formula
for the formal degree of a supercuspidal
in terms of its $L$-parameter only.
Our first main result is
to compute the formal degrees
of the supercuspidal representations
constructed by~Yu.
Our second, using the first,
is to verify that Kaletha's regular supercuspidal
$L$-packets satisfy the conjecture.
\end{abstract}

\maketitle

Let $G$ be a reductive algebraic group over
a nonarchimedean local field~$k$.
In the study of the representation theory
of the topological group $G(k)$,
the supercuspidal representations are fundamental:
there is a precise sense in which all
irreducible (smooth or unitary) representations
can be constructed from supercuspidals.
Much recent work has thus focused
on the construction and study of supercuspidal
representations.

In 2001, Yu \cite{yu01},
building on earlier work of Howe \cite{howe77}
and Adler \cite{adler98},
gave a general construction of supercuspidal representations
when $G$ splits over a tamely ramified extension of~$k$.
Six years later, Kim \cite{kim07} proved
that Yu's construction is exhaustive
when $k$ has characteristic zero
and the residue characteristic~$p$ of~$k$
is larger than some ineffective bound
depending on~$G$.
Recently, Fintzen \cite{fintzen18} has improved
Kim's result to show that Yu's construction
produces all supercuspidals when $p$
does not divide the order of the absolute
Weyl group.
Hence Yu's construction produces
all supercuspidals for many reductive groups,
though not all of them.
Moreover, the explicit nature of Yu's construction
makes his supercuspidals amenable to close study.

The collection of irreducible unitary representations
of~$G(k)$ carries a natural topology,
the Fell topology, and a natural Borel measure,
the Plancherel measure.
The Fell topology is canonical but
the Plancherel measure depends inversely
on a choice of Haar measure on~$G(k)$.
When $G$ is semisimple,
every supercuspidal representation~$\pi$ is unitary
and thus appears as a point of the unitary dual.
Since supercuspidal representations
of semisimple groups are discrete series,
this point is isolated.
We may thus ask for the measure of the point,
an interesting numerical invariant of~$\pi$
called the \emph{formal degree}.

When $G$ is not semisimple it is no longer 
necessarily the case that
all supercuspidal representations are unitary.
Nonetheless, we can define the formal degree
of an arbitrary supercuspidal representation
in a way that generalizes the formal degree
of a unitary supercuspidal.
The definition, given in \Cref{sec:aut:deg},
makes no reference to the unitary dual,
so we may forget about
the unitary representations of~$G(k)$
and focus our attention on the supercuspidal ones.

Our first main result is to compute
the formal degrees of Yu's supercuspidal
representations, \Cref{thm45}.
The formula uses some notation
that we must briefly recall
for its statement to be intelligible.
Yu's construction takes as input
a $5$-tuple~$\Psi$.
The first member of~$\Psi$ is an increasing sequence
$(G^i)_{0\leq i\leq d}$ of twisted Levi subgroups;
let $R_i$ denote the absolute root system of~$G^i$.
The second member is a certain point~$y$
in the Bruhat-Tits building~$\cal B(G)$.
The third member is an increasing
sequence $(r_i)_{0\leq i\leq d}$
of nonnegative real numbers.
The fourth member is a certain
irreducible representation~$\rho$
of the stabilizer $G^0(k)_{[y]}$
of the image~$[y]$ of~$y$ in~$\cal B^\tn{red}(G)$.
We compute the formal degree
with respect to a certain Haar measure~$\mu$
constructed by Gan and Gross
and discussed later in the introduction.
Gan and Gross's measure depends on a choice
of additive character
and in the formula we choose a level-zero character.
For simplicity of exposition
we assume in our discussion
of the formula that $G$ is semisimple,
though this assumption is relaxed
in the paper proper. 
Finally, let $\exp_q(t)\defeq q^t$.

\begin{theoremx}
\label{thm61}
Let $G$ be a semisimple $k$-group
and let $\Psi$ be a generic cuspidal $G$-datum
with associated supercuspidal representation~$\pi$.
Then
\[
\deg(\pi,\mu)
= \frac{\dim\rho}{[G^0(k)_{[y]}:G^0(k)_{y,0+}]}
\exp_q\biggl(
\tfrac12\dim G + \tfrac12\dim G^0_{y,0:0+}
+ \tfrac12\sum_{i=0}^{d-1} r_i(|\RS_{i+1}| - |\RS_i|)
\biggr).
\]
\end{theoremx}

The proof of \Cref{thm61} boils down,
after several reductions,
to computations in Bruhat-Tits theory.
Yu's supercuspidals are compactly induced from a certain
finite-dimensional irreducible representation~$\tau$
of a certain compact-open subgroup~$K$.
There is a general formula for the formal degree
of a compact induction which specializes,
in this case, to the ratio $\dim\tau/\!\vol(\mu,K)$.
The main difficulty is to compute the volume of~$K$.
We first situate $K$ as a finite-index subgroup
of a group of known measure;
this step reduces the problem to computing the index.
Using the theory of the Moy-Prasad filtration,
we can translate the computation of this index
into the computation of the length of a certain
subquotient of the Lie algebra.
The length computation, \Cref{thm18},
is our key technical result in the proof
of the formal-degree formula.

We can thus compute the formal degree of a broad class
of supercuspidal representations.
The other main result of the paper synthesizes
this computation with Langlands's
arithmetic parameterization of supercuspidals.

It is expected that the set~$\Pi(G)$
of smooth irreducible representations of~$G(k)$
is classified by certain homomorphisms
$\varphi:W_k'\to{}^L G$,
called \emph{$L$-parameters}.
Here $W_k$ is the Weil group of~$k$,
$W_k'\defeq \SL_2(\C)\times W_k$
is the Weil-Deligne group of~$k$,
and ${}^L G\defeq\widehat G\rtimes W_k$
is the (Weil form of the) $L$-group of~$G$.
The expected classification consists
of a partition
\[
\Pi(G) = \bigsqcup_\varphi \Pi_\varphi(G)
\]
of~$\Pi(G)$ into finite subsets~$\Pi_\varphi(G)$,
called \emph{$L$-packets},
indexed by (equivalence classes of)
$L$-parameters~$\varphi$.
The sets $\Pi_\varphi(G)$ are supposed to satisfy
many compatibility conditions,
the simplest of which are summarized
in Borel's Corvallis article
\cite[Section~10.3]{borel_corvallis2}.
The resulting partition of~$\Pi(G)$
is called a \emph{local Langlands correspondence}.
Although a local Langlands correspondence
has been constructed for many classes
of groups and representations,
the full correspondence remains a conjecture.
Even after fixing the group~$G$,
it is usually quite difficult
to establish the correspondence
for the entirety of~$\Pi(G)$.
Recent work has thus focused
on constructing the $L$-packets
of particular $L$-parameters.

Building on the outline of the correspondence,
Langlands suggested \cite[Chapitre~IV]{langlands83ias}
that the elements of the $L$-packet~$\Pi_\varphi(G)$
are parameterized by representations
of a certain finite group attached to~$\varphi$.
Refining Langlands's proposal,
Vogan enhanced \cite[Section~9]{vogan93}
$L$-parameters
to pairs $(\varphi,\rho)$ consisting of
an $L$-parameter~$\varphi$
and an irreducible representation~$\rho$ of the finite
group~$\pi_0(S_\varphi)$,
where $S_\varphi$ is the preimage
in~$\widehat G_\tn{sc}$
of the centralizer in~$\widehat G$ of~$\varphi$.
Unlike ordinary $L$-parameters,
these enhanced parameters keep
track of the inner class of~$G$:
one imposes an additional condition
\cite[Section~1]{hiraga_ichino_ikeda08}
on the central character of~$\rho$,
roughly, that it correspond
to the inner class of~$G$
via  Kottwitz's classification of inner forms.
In this formulation, it is expected
\cite[Section~1.2]{aubert_baum_plymen_solleveld18}
that the local Langlands correspondence
becomes a bijection, in other words,
that the irreducible representations~$\rho$ of $S_\varphi$
satisfying the central character condition
parameterize the $L$-packet~$\Pi_\varphi(G)$.

Using Vogan's enhanced $L$-parameters,
Hiraga, Ichino, and Ikeda predicted
\cite{hiraga_ichino_ikeda08,hiraga_ichino_ikeda08erratum}
that the formal degree of
an essentially discrete series representation
can be computed in terms of its $L$-parameter.
They verified the conjecture in many cases,
in particular, for real reductive groups
and for inner forms of $\SL_n$ and~$\GL_n$.
We will state their conjecture in a moment
after reviewing two of its inputs.

First, in a paper attaching motives to reductive groups,
Gross constructed \cite[Section~4]{gross97a}
a certain Haar measure~$\mu=\mu_\psi$ on~$G(k)$
depending on an additive character~$\psi$ of~$k$.
Two years later, he and Gan constructed
a closely related measure \cite[Section~5]{gan_gross99}
that conjecturally agrees with the original one.
Hiraga, Ichino, and Ikeda 
originally predicted that Gross's measure
was the right one for their conjecture,
but realized later \cite{hiraga_ichino_ikeda08erratum}
that one should use Gross and Gan's measure instead.

Second, one can attach to the parameter~$\varphi$
and certain additional data
-- a finite-dimensional representation~$r$ of~${}^L G$
and a nontrivial additive character~$\psi$ of~$k$ --
a meromorphic function $\gamma(s,\varphi,r,\psi)$
of the complex variable~$s$,
called a $\gamma$-factor of~$\varphi$.
The $\gamma$-factor is a product of $L$- and~$\varepsilon$-factors;
\cref{sec:gal:fac} recalls the precise formula.
For the representation~$r$ of~${}^LG$ we choose
the adjoint representation~$\tn{Ad}$ of ${}^LG$ on 
$\widehat{\frak g}/\,\widehat{\frak z}\,^{\Gamma\!_k}$,
where $\widehat{\frak g}$ and $\widehat{\frak z}$
are the Lie algebras of $\widehat G$ and of its center
and where $\Gamma\!_k$ is the absolute Galois group of~$k$.
Division by $\widehat{\frak z}\,^{\Gamma\!_k}$
ensures that the $\gamma$-factor
is defined at $s=0$.
We call the factor $\gamma(0,\varphi,\tn{Ad},\psi)$
appearing in the conjecture the
the \emph{adjoint $\gamma$-factor} of~$\varphi$.

We can now state the conjecture~%
\cite[Conjecture~1.4]{hiraga_ichino_ikeda08}
of Hiraga, Ichino, and Ikeda on the formal degree,
referred to in this paper as
the ``formal degree conjecture''
for the sake of brevity.
Let $S_\varphi^\natural$ denote the centralizer
of~$\varphi$ in~$\widehat{G^{\tn a}}$,
where $G^{\tn a}\defeq G/A$
with $A$ the maximal split central torus of~$G$.

\begin{conjecture}%
\label{thm1}
Let $\pi$ be an essentially discrete series
representation of~$G(k)$
with extended parameter~$(\varphi,\rho)$,
let $\psi$ be an additive character of~$k$,
and let $\mu_\psi$ be the Gross-Gan measure on~$G(k)$
attached to~$\psi$.
Then
\[
\deg(\pi,\mu_\psi) = \frac{\dim\rho}{|\pi_0(S_\varphi^\natural)|} 
\cdot |\gamma(0,\varphi,\tn{Ad},\psi)|.
\]
\end{conjecture}

Hiraga, Ichino, and Ikeda verified their conjecture in several cases,
building on the work of many others:
for an archimedean base field,
using Harish-Chandra's theory of discrete series
\cite{harish-chandra75};
for inner forms of $\GL_n$ and $\SL_n$,
using work of Silberger and Zink
\cite{zink93,silberger_zink96};
for some Steinberg representations,
using work of Kottwitz \cite{kottwitz88}
and Gross \cite{gross97a,gross97b};
for some unipotent discrete series
of adjoint split exceptional groups,
using work of Reeder \cite{reeder00};
and for some depth-zero supercuspidals
of pure inner forms of unramified groups,
using work of DeBacker and Reeder \cite{debacker_reeder09}.
In the years following the announcement of the conjecture,
Gan and Ichino \cite{gan_ichino14}
showed that it holds for $\tn U_3$,
$\tn{Sp}_4$, and $\tn{GSp}_4$;
Qiu \cite{qiu12} showed that it holds for the rank-one
metaplectic group;
Reeder and Yu \cite{reeder_yu14} and Kaletha \cite{kaletha15}
showed that it holds for epipelagic supercuspidals;
Gross and Reeder \cite{gross_reeder10}
showed that it holds for simple supercuspidals;
Ichino, Lapid, and Mao \cite{ichino_lapid_mao17}
showed that it holds for odd special orthogonal and metaplectic groups;
and Beuzart-Plessis \cite{beuzart-plessis20}
showed that it holds for unitary groups.

The formal degree conjecture is a ``meta-conjecture''
in the sense that it depends itself on a conjecture,
the local Langlands correspondence.
In order to verify the formal degree conjecture
one must first have access to
a candidate local Langlands correspondence,
or at least, to candidate $L$-packets.
Strictly speaking, the previous sentence is not entirely true because
some groups admit an analytic construction of the $\gamma$-factor
that bypasses the local Langlands correspondence,
though the two are expected to be compatible.
The main example is Godement-Jacquet's \cite{godement_jacquet72}
construction of the $L$- and $\varepsilon$-factors
for representations of the general linear group,
generalizing Tate's thesis.
Their construction explains how Hiraga-Ichino-Ikeda
were able to verify the formal degree conjecture
for the general linear group
using work that predated the Henniart \cite{henniart00}
and Harris-Taylor \cite{harris_taylor01}
constructions of the local Langlands correspondence.
Nonetheless, for the representations we consider in this paper,
an analytic theory of the $\gamma$-factor is not yet available,
and so we work with $L$-packets.

Recently, Kaletha has organized into $L$-packets
\cite{kaletha19a} most of Yu's supercuspidal representations,
the ``regular supercuspidal representations''.
His construction passes through
a pair~$(S,\theta)$ consisting
of an elliptic maximal torus~$S$ of~$G$
and a character~$\theta$ of~$S(k)$.
On the Galois side, one uses the Langlands-Shelstad
theory of $\chi$-data and extensions
of $L$-embeddings~\cite{langlands_shelstad87}
to construct a $L$-parameter for~$G$ from~$(S,\theta)$.
On the automorphic side, one uses
the pair~$(S,\theta)$ to produce an input
to Yu's construction,
hence a supercuspidal representation~$\pi$ of~$G(k)$.
We can thus interpret $\pi$ as a functorial lift of~$\theta$
with respect to the embedding $S\into G$.
The $L$-packet of~$\varphi$ consists, roughly,
of all $\pi$ produced in this way as we pass through
the various conjugacy classes of embeddings of~$S$ in~$G$;
\Cref{sec:llc} reviews the construction in more detail.

Since we can compute the formal degree
of Yu's representations, and Kaletha's $L$-packets
consist of such representations,
it is natural to ask whether the $L$-packets satisfy
the formal degree conjecture.
Our second main result,
proved in the body of the paper as \Cref{thm50},
is that they do.

\begin{theoremx}
\label{thm62}
Kaletha's regular $L$-packets satisfy
the formal degree conjecture.
\end{theoremx}

To prove \Cref{thm62},
we start by computing the adjoint representation
attached to a regular supercuspidal parameter:
it is a direct sum of the complexified character lattice of~$S$
and some monomial representations constructed
from~$\theta$ and the root system of~$S$.
The $\gamma$-factor of the character lattice has already
been computed in the literature.
As for the monomial representations,
computing their $\gamma$-factors amounts
to computing the depths of the inducing characters.
The inducing characters are very close
to certain characters naturally constructed from~$\theta$,
and whose depth is usually easy to understand;
the difficulty in the proof is to quantify
the difference between the two characters.
To quantify it, we prove that $\chi$-data
satisfy a natural base change formula, \Cref{thm32},
and that the inducing characters are ramified, \Cref{thm31}.

A refinement of the formal degree conjecture
due to Gross and Reeder \cite[Conjecture~8.3]{gross_reeder10}
predicts the root number of the adjoint representation.
In future work, I hope to determine whether
Kaletha's regular $L$-packets also satisfy
this refined conjecture.

The structure of this paper mirrors the formal degree conjecture.
After two preliminary sections that
fix notation and review the Langlands correspondence
for regular supercuspidals,
we compute the formal degree of a Yu representation
in \cref{sec:aut},
we compute the Galois side of
the formal degree conjecture in \cref{sec:gal},
and we compare the two
in the brief \cref{sec:com}.

\tableofcontents

\section{Notation}
\label{sec:not}

\subsection{Sets}
Let $|X|$ denote the cardinality of the set~$X$.
Given a subset $Y\subseteq X$,
let $\one_Y:X\to\{0,1\}$
denote the indicator function of~$Y$.

Many operations on sets are expressed
by superscripts or subscripts.
When we have several operations denoted this way,
say $X\mapsto X_a$ and $X\mapsto X_b$,
we use a comma to denote the concatenation:
$X\mapsto X_a\mapsto X_{a,b}$.
This expression is notationally simpler
than the longer form $(X_a)_b$
and should cause no confusion.

\subsection{Filtrations}
\label{sec:not:fil}
Suppose $I$ is a totally ordered index set
and $(X_i)_{i\in I}$ is a decreasing,
$I$-indexed filtration of the set~$X$.
For $i\in I$, define
\begin{equation}
\label{eq51}
X_{i+}\defeq\bigcup_{j>i} X_j.
\end{equation}
Let 
$\widetilde I\defeq I\cup\{i+:i\in I\}\cup\{\infty\}$
denote Bruhat and Tits's extension of~$I$
\cite[6.4.1]{bruhat_tits72};
their definition is for $I=\R$ only,
but it is clear how to extend it to arbitrary~$I$.
\Cref{eq51} together with the convention
\[
X_\infty \defeq \bigcap_i X_i
\]
defines an extension of the given filtration
to a $\widetilde I$-indexed filtration.
If in addition $X=G$ is a group
and each $X_i=G_i$ is a subgroup of~$G$
then define, for $i<j$ in~$\widetilde I$,
\[
G_{i:j} \defeq G_i/G_j.
\]

We apply this formalism to the Moy-Prasad filtration
on a $p$-adic group~$G(k)$ and its Lie algebra~$\frak g$
in \Cref{sec:not:bt},
and to the Weil group~$W_k$ in \Cref{sec:not:field}.
The filtrations on $G(k)$ and~$W_k$ are
indexed by $\R_{\geq0}$,
and the filtrations on~$\frak g$ are indexed by~$\R$.
In \Cref{sec:llc:pairs} we consider
a filtration (of a root system) that is increasing,
not decreasing.
When the filtration on~$X$ is increasing,
its extension to~$\widetilde I$
is defined by
\[
X_{i+}\defeq\bigcap_{j>i} X_j,
\qquad\qquad
X_\infty\defeq\bigcup_i X_i.
\]

\subsection{Fields}
\label{sec:not:field}
Let $k$ be a nonarchimedean local field
of odd residue characteristic~$p$,
let $\cal O$ denote the ring of integers of~$k$,
and let $\kappa$ denote the residue field of~$\cal O$.
Given a finite algebraic extension~$\ell$ of~$k$,
let $e_{\ell/k}$ denote the ramification degree
and $f_{\ell/k}$ the residue degree.

\begin{remark}
\label{thm63}
Many of the works this paper is built on,
for instance, Kaletha's construction of regular
supercuspidal $L$-packets \cite{kaletha19a},
assume that $p$ is odd.
For this reason we, too,
assume for the rest of the paper that $p$ is odd.
\end{remark}

Let $\ord_k:k^\times\to\Z$ denote
the unique discrete valuation on~$k$ with value group~$\Z$.
We extend $\ord_k$ to a valuation on
the algebraic closure~$\bar k$
and denote the extension by~$\ord_k$ as well.
Hence the value group
for a finite extension~$\ell$
is $\ord_k(\ell^\times)=e_{\ell/k}^{-1}\Z$.

\begin{remark}
Aesthetic reasons might lead one to consider a more general
value group for~$k$ than~$\Z$.
Indeed, per our convention,
the value group for
an extension of~$k$ is generally than larger than~$\Z$.
Most of the computations in this paper
can be modified to accommodate
a different choice of value group
because their defining object inherit that choice.
Many of the depth computations of \Cref{sec:aut}
could be modified to carry through because
the Moy-Prasad filtration inherits
its jumps from the value group.
Similarly, the Artin conductor computations
of \Cref{sec:gal:fac} could be modified to carry through
because the upper numbering filtration
inherits its jumps from the value group.
However, this modification would break
the connection between the Artin conductor
and the Artin representation.
Moreover, since the $\varepsilon$-factor is defined
independent of the value group,
the relationship between
the Artin conductor and the $\varepsilon$-factor,
\Cref{eq52}, holds only for value group~$\Z$.
\end{remark}

Given a finite $\cal O$-module~$M$,
let $\len M$ denote the length of~$M$.
When $k$ has positive characteristic
the module $M$ is a $\kappa$-vector space
and its length is its dimension,
but when $k$ has mixed characteristic
the module $M$ is \emph{not} a vector space,
and we must work instead with its length.

Let $q\defeq|\kappa|$, a power of~$p$,
let $\exp_q(t)\defeq t^q$,
and let $\log_q$ be the functional inverse of~$\exp_q$.
The function $\exp_q$ is related to the length
by the equation
\[
\exp_q\len M = |M|.
\]

Let $W_k$ denote the Weil group of~$k$,
let $I_k\subseteq W_k$ denote the inertia subgroup of~$W_k$,
let $P_k\subseteq I_k$ denote the wild inertia subgroup,
and for $r\geq0$ let $W_k^r\subseteq W_k$ denote
the $r$th subgroup of~$W_k$ in its upper numbering filtration,
computed with respect to the valuation~$\ord$.
A \emph{representation of the Weil group}
is a continuous, finite-dimensional, complex
representation~$\pi$ of~$W_k$.
Given a finite extension $\ell$ of~$k$,
let $\pi|_\ell\defeq\pi|_{W_\ell}$.
The \emph{depth} of~$\pi$ is defined as
\[
\depth\pi\defeq\inf\{r\in\R : \pi(W^r_{\!k})=1\}.
\]
In order to make the filtration on the Weil group
compatible with the Moy-Prasad filtration,
we need to modify the upper numbering filtration on~$W_\ell$
for a finite extension $\ell$ of~$k$ by
using the valuation $\ord_k$ to define it
instead of the valuation~$\ord_\ell$.
When the extension is tame, as is usually the case,
this modification has the effect of scaling
the indices of the filtration by $e_{\ell/k}^{-1}$.
To make the dependence on~$k$ clear,
let $\depth_k$ denote the depth
of a representation of~$W_\ell$
where its filtration is computed using~$\ord_k$.
The distinction is crucial for the proof of \Cref{thm12}.

\subsection{Groups}
Let $G$ be a reductive $k$-group,
let $Z$ be the center of~$G$,
and let $A$ be the maximal split subtorus of~$Z$.
We reserve the symbols $S$ and~$T$ for tori,
often maximal tori of~$G$.
Let $\frak g$ and $\frak z$ denote
the Lie algebras of $G$ and~$Z$, respectively.

\begin{remark}
There are three exceptions to the notational convention
that $G$ denotes a reductive group
and $S$ and $T$ denotes tori.

First, we sometimes need to work with $\kappa$-groups
instead of $k$-groups.
This practice cannot be avoided,
but is so rare that we did not see the value
in introducing a separate notational convention
for~$\kappa$-groups.
So $G$, $S$, and $T$ denote $\kappa$-groups
in this setting, which takes place in small portions
of \Cref{sec:llc:reps,sec:aut:reg}.

Second, for reasons of notational clarity,
in \Cref{sec:aut} the symbol $\ul G$ denotes a $k$-group
and the symbol $G$ denotes
the topological group of its rational points,
as discussed in \Cref{sec:aut:not}.

Third, in \Cref{sec:aut:deg},
where we discuss the formal degree,
$G$ denotes an arbitrary locally profinite group,
the proper setting for that theory.
\end{remark}

Given a subgroup $H$ of~$G$,
let $H^{\tn a}\defeq G/A$.
The letter a abbreviates ``anisotropic''.
The notation hides the dependence
on the ambient group~$G$,
but the meaning should be clear from context:
typically $H$ is a maximal torus
or (twisted) Levi subgroup of~$G$.

Let $\widehat G$ denote the Langlands dual group of~$G$,
a complex reductive group,
and let ${}^LG=\widehat G\rtimes W$
denote the Weil form of the dual group.

Given an extension $\ell$ of~$k$
and an $\ell$-group~$H$,
let $\Res_{\ell/k}H$ denote the Weil restriction
of~$H$ from~$\ell$ to~$k$.

\subsection{Root systems}
Given a reductive $k$-group~$G$
and a maximal torus~$T$ of~$G$,
let $\RS(G,T)$ be the absolute root system
of~$G$ with respect to~$T$,
that is, the root system of~$G_{\bar k}$
with respect to~$T_{\bar k}$,
together with its natural Galois action.
Let $\underline\RS(G,T)$
denote the set of Galois orbits of~$\RS(G,T)$.
I prefer to think of~$\underline\RS(G,T)$
as the ``functor of roots'' in the sense
of SGA~3 \cite[XIX.3]{sga3new},
an equivalent but more elaborate perspective.
Reserve the letters $\alpha,\beta,\gamma,\dots$
for elements of~$\RS$, and their underlines
$\ul\alpha,\ul\beta,\ul\gamma,\dots$
for elements of~$\ul\RS$.
Given $\ul\alpha\in\underline\RS$,
let $\ul\alpha(\bar k)$ denote
the elements of~$\ul\alpha$, a subset of~$\RS$.

Given $\alpha\in\RS$,
let $\Gamma\!_\alpha$ denote the stabilizer
of~$\alpha$ in~$\Gamma\!_k$
and let $k_\alpha\defeq\bar k^{\Gamma\!_\alpha}$
denote the fixed field of~$\Gamma\!_\alpha$.
Given $\ul\alpha\in\RS$,
let $k_{\ul\alpha}$ denote a fixed field extension of~$k$
that is isomorphic,
for some $\alpha\in\ul\alpha(\bar k)$,
to the extension $k_\alpha$,
and let $\kappa_{\ul\alpha}$
denote the residue field of $k_{\ul\alpha}$.
We can define $k_{\ul\alpha}$
canonically as the inverse limit
of the groupoid of extensions~$k_\alpha$
with $\alpha\in\alpha(\bar k)$,
in the style of Deligne and Lusztig
\cite[Section~1.1]{deligne_lusztig76},
but since $\Gamma\!_k$ is nonabelian,
it is impossible to canonically identify
this limit with any one~$k_\alpha$.
The notation $k_{\ul\alpha}$ helps
us avoid choosing such an $\alpha$,
and is reserved for expressions that depend only on
the isomorphism class of the extension,
such as the degree $[k_{\ul\alpha}:k]$.

\subsection{Bruhat-Tits theory}
\label{sec:not:bt}
Given a reductive $k$-group~$G$,
let $\cal B(G)$ and~$\cal B^\tn{red}(G)$
be the extended and reduced 
Bruhat-Tits buildings of~$G$.
Let $[x]$ denote the image of~$x$
under the canonical reduction map
$\cal B(G)\to\cal B^\tn{red}(G)$.
Given a split maximal torus~$T$ of~$G$,
let $\cal A(G,T)$ and $\cal A^\tn{red}(G,T)$
denote the extended and reduced apartments of~$T$.
The apartment $\cal A(G,T)$ is noncanonically isomorphic
to the building $\cal B(T)$.
For us, the words ``building''
and ``apartment'' refer to the extended forms.

The apartment of a split maximal torus
is a classical construction,
defined in Bruhat and Tits's original papers on buildings
\cite{bruhat_tits72,bruhat_tits84}.
More recently, in a paper establishing
tame descent for buildings \cite{prasad_yu02},
Prasad and Yu showed
for any tame maximal torus~$S$ of~$G$
how to embed the building of~$S$
into the building of~$G$.
Although the embedding is not canonical,
the image of the embedding is canonical.
We denote this image by $\cal A(G,S)$
and call it the \emph{apartment} of~$G$ in~$S$,
though that terminology is typically reserved
for maximal split tori only.

Given a point $x$ of~$\cal B(G)$
or~$\cal B^\tn{red}(G)$,
let $G(k)_x$ denote the stabilizer
of~$x$ in~$G(k)$.
When the center of $G$ is anisotropic
the parahoric group $G(k)_{x,0}$
is of finite index in the stabilizer~$G(k)_x$;
in general, $A(k)_0G(k)_{x,0}$
is of finite index in~$G(k)_x$.

For each $x\in\cal B(G)$,
Moy and Prasad defined
\cite[Sections~2 and 3]{moy_prasad94}
a canonical decreasing $\R_{\geq0}$-indexed filtration of~$G(k)_x$
and a canonical decreasing $\R$-indexed filtration of~$\frak g$,
denoted by $G(k)_{x,r}$ and $\frak g_{x,r}$
and called the \emph{Moy-Prasad filtrations}.
When $G=S$ is a torus the point $x$
is irrelevant and we suppress it from the notation.
Although the filtrations are defined for every~$G$,
they are particularly well-behaved
when $G$ splits over a tame extension.
In this case, for instance,
there is a canonical isomorphism
\[
G_{x,r:r+} \simeq \frak g_{x,r:r+}
\]
for every $r>0$,
called the \emph{Moy-Prasad isomorphism}.
\Crefrange{sec:aut:split}{sec:aut:yu_grp}
discuss in much more detail
these filtrations and a generalization of them
due to Yu.

The Moy-Prasad filtration is compatible with our chosen
discrete valuation on~$k$ in the sense that for $r>0$,
\[
k^\times_r \defeq \{a\in k^\times : \ord_k(a-1)\geq r\}.
\]
Given a finite separable extension $\ell$ of~$k$
and a reductive $\ell$-group~$H$,
we compute the Moy-Prasad filtration on~$H(\ell)$
with respect to the norm~$\ord_k$, not $\ord_\ell$.
This convention implies that the Moy-Prasad filtration
is a topological invariant independent
of the base field in the following sense:
the Moy-Prasad filtration on the group~$G(\ell)$
agrees with the Moy-Prasad filtration
on the identical group $(\Res_{\ell/k}G)(k)$.

The \emph{depth} of an irreducible
admissible representation $\pi$ of~$G(k)$,
denoted by $\depth_k\pi$,
is the smallest real number~$r$
such that for some $x\in\cal B(G)$,
the representation~$\pi$ has a nonzero vector
fixed by $G(k)_{x,r+}$.
The subscript~$k$ is a reminder
that the depth, via the Moy-Prasad filtration
by which it is defined, depends on the base field~$k$.
It is not a priori clear that this minimum is attained, but
Moy and Prasad showed \cite[Theorem 3.5]{moy_prasad96} that
the depth is a nonnegative rational number.
This result makes the depth an indispensable tool
in the representation theory of $p$-adic groups.

\begin{remark}
Most of Bruhat-Tits theory carries through
when the field~$k$ is assumed only to be Henselian.
For example, the results of
\Cref{sec:aut:split,sec:aut:tame,sec:aut:yu_grp,sec:aut:len}
hold at this level of generality.
But once representation theory enters the picture,
we must assume $k$ is a local field.
\end{remark}

\subsection{Base change for groups and vector spaces}
Given a scheme $X$ over a base scheme~$Y$
and a morphism $Z\to Y$,
let $X_Z$ denote the base change of~$X$
from $Y$ to~$Z$, that is,
the pullback $X\times_Y Z$.
When $Z=\Spec A$ is the spectrum of a field~$A$,
we write $X_A$ for~$X_Z$.
The schemes $X$ that we base change
are in practice always group schemes.

Similar notation can be used for base change of modules.
Given a $B$-algebra~$A$ and an $A$-module~$M$,
let $M_B\defeq M\otimes_AB$.

We take the position that an algebraic group
carries the information of its base scheme.
This position forces our terminology
to differ slightly from common practice
in the literature where a $k$-group
is thought of as a $\bar k$-group
with a $k$-rational structure.
In this common language one can speak,
given two $k$-groups $G$ and $H$,
of morphisms $G\to H$ that are
not defined over~$k$.
For us, a morphism $G\to H$
is automatically defined over~$k$.
To speak of a morphism ``not defined over $k$''
in this common sense,
we would speak of a morphism $G_\ell\to H_\ell$
where $\ell$ is an extension of~$k$,
especially $\ell=\bar k$.

\subsection{Base change for characters}
\label{sec:not:base}
Let $S$ be a $k$-torus,
let $\ell$ be a finite separable extension of~$k$,
and let $T\defeq\Res_{\ell/k}S_\ell$.
Since $X^*(T)=\Ind_{\ell/k}\Res_{\ell/k}X^*(S)$,
there is a canonical map $X^*(S)\to X^*(T)$
of Galois lattices, the unit of 
the adjoint pair $(\Res_{\ell/k},\Ind_{\ell/k})$.
The dual of this unit
is a canonical map $N_{\ell/k}:T\to S$,
called the \emph{norm map};
we use the same name and notation for the map
$T(k)=S(\ell)\to S(k)$ on rational points.
Given a character $\theta:S(k)\to\C^\times$,
define the character
$\theta_{\ell/k}:S(\ell)\to\C^\times$ by
precomposition with the norm:
\[
\theta_{\ell/k}\defeq \theta\circ N_{\ell/k},
\]
We call $\theta_{\ell/k}$ the \emph{base change}
of~$\theta$ from $k$ to~$\ell$.
Contrary to the usual notation
for base change of schemes,
the notation for base change of characters
must include $k$, not just~$\ell$,
because the base field~$k$
cannot be recovered
from the topological group~$S(k)$.

The base change operation for characters 
realizes base change in the local Langlands correspondence.
For tori, this correspondence is a bijection between
the complex character group of~$S(k)$
and the Galois cohomology group $\tn H^1(W_k,\widehat S)$.
On the Galois side,
we can restrict the $L$-parameter~$\hat\theta$
of a character~$\theta$ to the Weil group $W_\ell$
of a finite separable extension.
But by the local Langlands correspondence for~$S_\ell$,
this parameter~$\hat\theta|_\ell$
corresponds to a character of~$S(\ell)$.
It is well known, and a formal consequence
of the properties of the local Langlands
correspondence for tori, that this character
is precisely the base-changed character just defined:
symbolically,
\[
\hat\theta|_\ell = \widehat{\theta_{\ell/k}}.
\]

Yu's Ottawa article \cite{yu_ottawa2} nicely summarizes
the local Langlands correspondence for tori,
and proves that for tame tori,
the local Langlands correspondence
preserves depth~\cite[Section~7.10]{yu_ottawa2}.
He does not discuss base change, however.

We also need to understand
how base change affects depth.

\begin{lemma}
\label{thm39}
Let $\ell$ be a finite separable extension of~$k$,
let $S$ be a $k$-torus,
and let $\theta:S(k)\to\C^\times$ be a character.
If either
\begin{enumerate}
\item
$\ell/k$ is unramified and $\depth_k\theta\geq0$ or
\item
$\ell/k$ tamely ramified and $\depth_k\theta>0$
\end{enumerate}
then $\depth_k\theta = \depth_k\theta_{\ell/k}$.
\end{lemma}

\begin{proof}
Using the same trick as in Yu's proof
of the depth-preservation theorem,
it suffices to prove the \lcnamecref{thm39}
in the case where $S=\G_{\tn m}$.
The first case is an immediate consequence
of local class field.
Using the first result,
reduce the second to the case
where $\ell$ is obtained from~$k$
by adjoining a root of a uniformizer;
from here, the result is a straightforward computation.
\end{proof}

\section{Langlands correspondence for regular supercuspidals}
\label{sec:llc}
In this section we review the Langlands correspondence
for regular supercuspidal representations,
following Kaletha's article \cite{kaletha19a}.
Many of the definitions, for instance,
regularity of $L$-parameters, are rather technical,
and instead of restating them,
we point to their definitions in the literature.
The description of regular representations
and the construction of their $L$-parameters
passes through a pair $(S,\theta)$
consisting of an elliptic maximal torus~$S$ of~$G$
and a character~$\theta$ of~$S(k)$
satisfying certain regularity conditions
reviewed in \Cref{sec:llc:pairs}.
The primary goal of this section, then,
is to understand, to the extent needed
to verify the formal degree conjecture,
how these pairs interface
with both sides of the Langlands correspondence.

On the automorphic side, the pair $(S,\theta)$
produces an input to Yu's construction~\cite{yu01}
of supercuspidals; 
we explain how this works in \Cref{sec:llc:reps}.
In this way we produce
a supercuspidal representation~$\pi_{(S,\theta)}$ of~$G(k)$.
When $(S,\theta)$ is ``tame elliptic regular'',
the representations that arise this way are precisely
the regular supercuspidal representations.

On the Galois side, one can define a certain class
of ``regular supercuspidal parameters''
and show that each arises from a pair~$(S,\theta)$
as the composition
\[
W_k \xto{{}^L\theta} {}^LS \xto{{}^Lj_\chi} {}^L G.
\]
Here the first map corresponds to~$\theta$
under the local Langlands correspondence for tori
and the second map is an extension
of a given Galois-stable embedding
$\hat\jmath:\widehat S\to\widehat G$.
There is a general procedure,
reviewed in \Cref{sec:llc:emb},
for extending $\hat j$ to ${}^Lj_\chi$ using
a certain object~$\chi$ called
a set of $\chi$-data.
In our setting one canonically constructs
such $\chi$-data from the pair~$(S,\theta)$,
producing a canonical extension~${}^Lj_\chi$
and thus a canonical $L$-parameter.
Using pairs $(S,\theta)$,
we organize regular supercuspidal parameters
into $L$-packets in \Cref{sec:llc:params}.
To first approximation a regular
supercuspidal $L$-packet consists
of the regular supercuspidal representations
$\pi_{(jS,\theta\circ j^{-1})}$ as
$j$ ranges over $G(k)$-conjugacy classes
of admissible embeddings $j:S\into G$;
in reality, however,
we must slightly modify
the character~$\theta\circ j^{-1}$.

\subsection{Tame elliptic regular pairs}
\label{sec:llc:pairs}
Pairs $(S,\theta)$, consisting of a $k$-torus~$S$ and
a character $\theta:S(k)\to\C^\times$
subject to certain conditions,
mediate the local Langlands correspondence
for regular supercuspidals.
This subsection reviews these conditions,
following the discussion in Kaletha's article
\cite[Sections 3.6 and 3.7]{kaletha19a},
and uses them to compute the depths
of certain auxiliary characters
that arise in \Cref{sec:gal:root}.

The simplest condition is tameness:
the pair $(S,\theta)$ is \emph{tame}
if $S$ is tame, that is, if $S$ splits over
a tamely ramified extension of~$k$.
In this paper $S$ is always assumed tame,
though we sometimes repeat the hypothesis for emphasis.

All other conditions on our pair
require that $S$ be embedded
as a maximal torus of a reductive group~$G$.
This requirement is extremely natural 
on the automorphic side, but on the Galois side,
we must reinterpret it carefully
since the embedding is allowed to vary.

So assume in the rest of this subsection
that $S$ is a maximal torus of~$G$.
The pair $(S,\theta)$ is \emph{elliptic}
if $S$ is elliptic, that is, if the torus~$S/Z$
(where $Z$ is the center of~$G$) is anisotropic.

For the final condition, regularity,
we need to probe more deeply the relationship
between $(S,\theta)$ and~$G$.
Let $\RS=\RS(G,S)$
and let $\ell$ be the splitting field of~$S$.
Given a real number~$r>0$, consider the set of roots
\[
\RS^r \defeq \{\alpha\in\RS \mid
(\theta_{\ell/k}\circ\alpha^\vee)(\ell_r^\times) = 1\}.
\]
The assignment $r\mapsto\RS^r$
is an increasing, Galois-stable,
$\R$-indexed filtration of~$\RS$.
Let $\RS_{r+}\defeq\bigcup_{s>r}\RS_s$,
let $r_{d-1} > \cdots > r_0 > 0$ be the breaks of this filtration,
and let $r_{-1} = 0$ and $r_d=\depth_k\theta$.
For each $0\leq i\leq d$, let $\RS_i\defeq \RS^{(r_{i-1})+}$.
It turns out~\cite[Lemma~3.6.1]{kaletha19a}
that $\RS_i$ is a Levi subsystem of~$\RS$.
Let $G^i$ be the connected reductive subgroup of~$G$
containing~$S$ whose root system
with respect to~$S$ is $\RS_i$;
the group~$G^i$ can be constructed
by Galois descent, for example.
Let $G^{-1}\defeq S$.

\begin{definition}[{\cite[Definition 3.7.5]{kaletha19a}}]
\label{thm70}
A tame elliptic pair is \emph{regular} if
\begin{enumerate}
\item
the action of the inertia group on~$\RS_0$
preserves a set of positive roots, and
\item
the stabilizer of the action of the group
$N(G^0,S)(k)/S(k)$
on $\theta|_{S(k)_0}$ is trivial.
\end{enumerate}
It is \emph{extra regular} if, in addition,
\begin{enumerate}
\item[($2'$)]
the stabilizer of the action of the group
$\Omega(G^0,S)(k)$ on $\theta|_{S(k)_0}$ is trivial.
\end{enumerate}
Here $N(G^0,S)$ is the normalizer of $G^0$ in~$S$
and $\Omega(G^0,S)$ is the Weyl group.
\end{definition}

When we compute in \Cref{sec:gal}
the formal degree of the regular
parameter attached to~$(S,\theta)$,
half of the problem
(the ``root summand'' of \Cref{sec:gal:root})
boils down to knowing for each coroot~$\alpha^\vee$
the depth of the character
\[
\theta_{k_\alpha/k}\circ\alpha^\vee;
\]
here $\alpha^\vee$ is interpreted
as a homomorphism $k_\alpha^\times\to S(k_\alpha)$.
Therefore, our main goal in this subsection
is to compute the depth of this character.
To carry out the computation,
we systematically decompose $\theta$
as a product of characters of known depth
using Kaletha's notion of a Howe factorization,
after reviewing an important component
of that definition, due to~Yu.

The definition of a Howe factorization relies
on a definition of Yu \cite[Section~9]{yu01}
for a character~$\phi:H(k)\to\C^\times$
of a twisted Levi subgroup~$H$ of~$G$ to be
\emph{$G$-generic of depth~$r$}.
We need not concern ourselves with
the precise definition of $G$-genericity,
but we do need one of its consequences,
which approximates the full definition.

\begin{lemma}
\label{thm21}
Let $G$ be a reductive $k$-group,
let $H$ be a tame twisted Levi subgroup of~$G$,
let $S\subseteq H$ be a tame maximal torus,
and let $\phi:H(k)\to\C^\times$ be
a character of positive depth~$r$
whose restriction to $H_\tn{sc}(k)$ is trivial.
Then $\phi$ is $G$-generic if and only if
for every root $\alpha\in\RS(G,S) \setminus\RS(H,S)$
and every finite tame extension $\ell$ of~$k_\alpha$,
the character $\phi_{\ell/k}\circ\alpha_\ell^\vee$
of $\ell^\times$ has depth~$r$.
\end{lemma}

\begin{proof}
Kaletha proved this \lcnamecref{thm21}
in the case where $\ell$ is a fixed
splitting field of~$S$ \cite[Lemma~3.6.8]{kaletha19a}.
We can deduce our result in the case where $\ell=k_\alpha$
from his result using the naturality of the norm map,
and from there, we can deduce the result in general
using naturality and \Cref{thm39}.
\end{proof}

\begin{corollary}
\label{thm22}
In the setting of \cref{thm21},
$\depth_k\phi\in\ord(k_\alpha^\times)$
for each $\alpha\in\RS(G,S)\setminus\RS(H,S)$.
\end{corollary}

Regularity is much less restrictive than genericity,
but we need to know something about genericity
in order to understand the depths
of various auxiliary characters constructed from~$\theta$
in \Cref{sec:gal:root}.
Roughly speaking, any character, regular or not,
can be decomposed as a product of generic characters
related to the filtration of the root system.
This decomposition is called a Howe factorization.

\begin{definition}
A \emph{Howe factorization} of~$(S,\theta)$
is a sequence 
$(\phi_i:G^i(k)\to\C^\times)_{-1\leq i\leq d}$
of characters satisfying the following properties.
\begin{enumerate}
\item
$\displaystyle\theta = \prod_{i=-1}^d \phi_i|_{S(k)}$.
\item
The character~$\phi_i$ is trivial on~$G^i_\tn{sc}(k)$
for $0\leq i\leq d$.
\item
The character~$\phi_i$ is $G^{i+1}$-generic of depth~$r_i$
for $0\leq i\leq d-1$;
the character~$\phi_d$ is trivial if $r_d = r_{d-1}$
and has depth~$r_d$ otherwise;
and the character~$\phi_{-1}$ is trivial if $G^0=S$
and otherwise satisfies $\phi_{-1}|_{S(k)_{0+}}=1$.
\end{enumerate}
\end{definition}

It turns out~\cite[Proposition~3.6.7]{kaletha19a}
that every tame pair admits a Howe factorization.
When $\alpha\notin\RS_0$,
we can use this factorization
to compute the depth of
$\theta_{k_\alpha/k}\circ\alpha^\vee$.

\begin{lemma}
\label{thm38}
Let $(S,\theta)$ be a tame pair
and let $\alpha\in\RS_i$, where $1\leq i\leq d$.
Then the character $\theta_{k_\alpha/k}\circ\alpha^\vee:
k_\alpha^\times\to\C^\times$ has depth~$r_{i-1}$.
\end{lemma}

\begin{proof}
Let $(\phi_j:G^j(k)\to\C^\times)_{-1\leq j\leq d}$
be a Howe factorization of~$(S,\theta)$.
Then
\[
\theta_{k_\alpha/k}\circ\alpha^\vee
= \prod_{j=-1}^d \phi_{j,k_\alpha/k}\circ\alpha^\vee. 
\]
Since $\alpha^\vee$ factors through $G^i$,
condition (2) of a Howe factorization implies that
the factors of this product are trivial for $j\geq i$.
By \Cref{thm21}, the $j$th remaining factor has depth~$r_j$,
and since the sequence $j\mapsto r_j$
is strictly increasing,
the product has depth~$r_{i-1}$.
\end{proof}

\Cref{thm38} conspicuously omits
the case where $\alpha\in\RS_0$.
We have more to say about this in
\Cref{sec:gal:root},
especially in \Cref{thm31,thm73}.

\subsection{Regular representations}
\label{sec:llc:reps}
Yu's seminal paper~\cite{yu01} constructs a broad class
of supercuspidal representations
starting from a certain triple $(\vec G, \pi_{-1},\vec\phi)$,
which we call, for reference, a \emph{Yu datum}.
This subsection reviews these triples
and explains how a tame elliptic regular pair
gives rise to a Yu datum.
Later, \Cref{sec:aut:yu} explains
in more detail how to construct
supercuspidal representations from Yu data.

There are three stages of representations 
used in in Yu's construction,
each informing the previous one:
representations of finite groups of Lie type;
depth-zero supercuspidal representations;
and supercuspidal representations of arbitrary depth.
Since the definition of regular supercuspidal
passes through each of these stages,
we start by reviewing each stage in turn.

The first stage is finite groups of Lie type.
In this paragraph only,
let $G$ be a reductive $\kappa$-group.
Representations of~$G(\kappa)$ are well understood
through the work of Deligne and Lusztig
\cite{deligne_lusztig76}.
They attached to an elliptic pair~$(S,\theta)$ over~$\kappa$
(that is, $S$ is an elliptic maximal torus of~$G$)
a virtual representation $\tn R_{(S,\theta)}$,
the Deligne-Lusztig induction.
If the character~$\theta$ of the $\kappa$-torus~$S$
is \emph{regular}%
~\cite[Definition~3.4.16]{kaletha19a}
then $\pm\tn R_{(S,\theta)}$
is an irreducible representation.
We say that a representation~$\rho$ of~$G(\kappa)$
is \emph{regular} if it is isomorphic to such
a Deligne-Lusztig representation.

Passing to the second stage,
depth-zero supercuspidals,
the result that initiated general constructions
of supercuspidals was the following classification
theorem \cite[Proposition~6.8]{moy_prasad96}
of Moy and Prasad:
for every depth-zero supercuspidal
representation~$\pi$ of~$G(k)$,
there is a vertex $x\in\cal B^\tn{red}(G)$
such that $\pi|_{G(k)_{x,0}}$
contains the inflation to~$G(k)_{x,0}$
of an irreducible cuspidal representation~$\widetilde\rho$
of~$G(k)_{x,0:0+}$.
Furthermore, we can recover $\pi$
by compact induction:
there is some representation $\rho$
of $G(k)_x$ such that 
$\rho|_{G(k)_{x,0}}$ contains
the inflation of $\widetilde\rho$
and such that $\pi=\cInd_{G(k)_x}^{G(k)}\rho$.
The depth-zero supercuspidal representation~$\pi$
is \emph{regular} if the representation~$\rho$
of~$G(k)_{x,0:0+}$ is regular.
Regular depth-zero supercuspidals~$\pi$
enjoy two pleasant properties.

First, there is a canonical bijection between
regular depth-zero supercuspidals
and conjugacy classes 
of regular tame elliptic pairs~$(S,\theta)$
of depth zero in which the torus~$S$
is \emph{maximally unramified} in~$G$
\cite[Definition 3.4.1]{kaletha19a}.
In particular, we can recover~$S$ from~$\pi$.

Second, it turns out%
~\cite[Sections 3.4.4 and 3.4.5]{kaletha19a}
that $\pi_{(S,\theta)}$ can be compactly induced from
a representation~$\eta_{(S,\theta)}$
of the group $S(k)G(k)_{x,0}$.
This is an improvement over Moy and Prasad's
theorem, which uses the larger 
stabilizer group $G(k)_x$ instead;
generally $S(k)G(k)_{x,0}$ is easier
to understand than $G(k)_x$.
The fact that $\pi$ is compactly induced
from this smaller group
plays a crucial role in our final computation,
in \Cref{sec:aut:reg},
of the formal degree of a regular supercuspidal.

We can now discuss the third stage,
Yu's general construction of supercuspidals.

To start, we recall the definition of a Yu datum.
A subgroup $H$ of~$G$ is a \emph{twisted Levi}
subgroup if there is a finite separable
extension $\ell$ of~$k$ splitting~$G$
such that $H_\ell$ is a Levi subgroup of~$G_\ell$,
and $H$ is \emph{tame} if $\ell$
can be taken to be a tame extension of~$k$.
A \emph{twisted Levi sequence} in~$G$
is an increasing sequence
\[
\vec G = (G^0\subsetneq G^1
\subsetneq\cdots\subsetneq G^d)
\]
of twisted Levi subgroups of~$G$;
it is \emph{tame} if each of its members is tame.
The first component $\vec G$ of a Yu datum
is a tame twisted Levi sequence;
the second component~$\pi_{-1}$ of a Yu datum
is a depth-zero supercuspidal representation
of~$G^0(k)$; and
the third component~$\vec\phi$ of a Yu datum
is a sequence of characters
\[
\vec\phi = (\phi_i:G^i(k)\to\C^\times)_{0\leq i\leq d}.
\]
These three objects are required to
satisfy certain conditions that
\Cref{sec:aut:yu} spells out in detail.
In fact, in that section we work
with a certain five-tuple instead of a Yu datum,
but the two objects are closely related
\cite[Section~3.1]{hakim_murnaghan08}.

To simplify the following definition,
we assume in the rest of this subsection
that $p$ does not divide the order of 
the fundamental group of~$G$.
Kaletha defined regularity in general
using $z$-extensions \cite[Section~3.7.4]{kaletha19a},
but we have no need to understand how this works.

\begin{definition}
A Yu datum $(\vec G,\pi_{-1},\vec\phi)$
is \emph{regular} if $\pi_{-1}$ is 
a regular depth-zero supercuspidal representation.
A supercuspidal representation is \emph{regular}
if it is isomorphic to a supercuspidal representation
constructed from a regular Yu datum.
\end{definition}

We have thus defined the supercuspidal
representations of interest to us;
the next matter is to connect them
to torus-character pairs.

Given a Yu datum $(\vec G,\pi_{-1},\vec\theta)$,
we can find a maximally unramified torus $S$ of~$G^0$
and a regular depth-zero character~$\phi_{-1}$ of~$S(k)$
such that $\pi_{-1}=\pi_{(S,\phi_{-1})}$.
Setting
\[
\theta = \prod_{i=-1}^d \phi_i|_{S(k)}
\]
then produces a tame elliptic regular pair~$(S,\theta)$.
Conversely, given such a pair $(S,\theta)$,
with Howe factorization
$(\phi_i:G^i(k)\to\C^\times)_{-1\leq i\leq d}$,
the triple 
\[
(\vec G=(G^i)_{0\leq i\leq d},
\pi_{-1}=\pi_{(S,\phi_{-1})},
\vec\phi=(\phi_i)_{0\leq i\leq d})
\]
is a Yu datum.
It turns out \cite[Proposition 3.7.8]{kaletha19a}
that these assignments are bijections
modulo the appropriate equivalences.
In this way, we can form a regular
supercuspidal representation $\pi_{(S,\theta)}$
from a tame elliptic regular pair~$(S,\theta)$.

\subsection{\texorpdfstring{$L$}{L}-embeddings}
\label{sec:llc:emb}
In this subsection we explain and study a formalism
of Langlands and Shelstad for extending
an embedding $\hat\jmath:\widehat S\to\widehat G$
with Galois-stable $\widehat G$-conjugacy class
to an $L$-embedding ${}^L j:{}^L S\to{}^L G$.
Fix a $\Gamma\!_k$-pinning of~$\widehat G$
with maximal torus~$\widehat T$.

The first difficulty in extending $\hat\jmath$
to ${}^L j$ is to reconcile
the Galois actions on~$\widehat S$ and~$\widehat G$.
Specifically, let $\tau_G$ denote the action
(homomorphism) of~$W_k$ on~$\widehat T$
through its action on~$\widehat G$;
let $\tau_S$ denote the action of~$W_k$
on~$\widehat T$ through its action on~$\widehat S$,
transferred using~$\hat\jmath$;
let $N$ be the normalizer of~$\widehat T$ in~$\widehat G$;
and let $\Omega=\Omega(\widehat G,\widehat T)$ be the Weyl group.
Given a Weil element $w\in W_k$,
thought of as an element of~${}^L S$,
its image under an extension~${}^L j$
has the form $nw$ where $n\in N$
lifts the Weyl element
\[
\omega_{S,G}(w)\defeq \tau_S(w)\tau_G(w)^{-1}\in\Omega,
\]
so that $nw$ acts on~$\widehat T$ by
the $\widehat S$-action.
The lift exists precisely because
the $\widehat G$-conjugacy class of $\hat\jmath$
is Galois-stable.
To define the extension ${}^L j$, then,
we need only choose a specific lift~$n$
of~$\omega_{S,G}(w)$ to~$N$.

For many reductive groups,
in particular, the special linear group,
the projection map $N\to\Omega$
does not admit a homomorphic section.
Nonetheless, by finding a good way to lift
fundamental reflections, Tits~\cite{tits66} defined
a canonical set-theoretic section $n:\Omega\to N$,
which we call the \emph{Tits lift}.
The precise definition of the lift
\cite[Section~2.1]{langlands_shelstad87}
is not so important for us.

Since the Tits lift is not a homomorphism,
the candidate formula
\[
w\mapsto n(\omega_{S,G}(w))w
\]
for ${}^L j|_{W_k}$ does not define a homomorphism
$W_k\to{}^L G$.
To get around this problem,
Langlands and Shelstad studied the failure
of this formula to define a homomorphism
as measured by the function
$t:W_k\times W_k\to\widehat T$ given by
\[
t(w_1,w_2) \defeq 
(n(w_1)w_1)(n(w_2)w_2)(n(w_1w_2)w_1w_2)^{-1}.
\]
Using an object~$\chi$ called a set of $\chi$-data,
whose definition we review momentarily,
they constructed a function $r_\chi:W_k\to\widehat S$
that negates this failure in the sense that
$\partial(\hat\jmath\circ r_\chi) = t^{-1}$,
where $\partial$ is the coboundary operator.
Hence the modified formula
\[
w\mapsto\hat\jmath(r_\chi(w))n(\omega_{S,G}(w))w
\]
does define a homomorphism $W_k\to{}^L G$,
and in total, the extension ${}^L j_\chi$
of~$\hat\jmath$ is given by the formula
\[
{}^L j_\chi(sw)
= \hat\jmath(s\,r_\chi(w))n(\omega_{S,G}(w))w.
\]

This concludes our outline of the Langlands-Shelstad
procedure for extending an embedding of a torus
to an $L$-embedding.
Later, in \Cref{sec:gal:chi},
we recall the definitions of $\chi$-data
and the function~$r_\chi$.

\subsection{Regular parameters}
\label{sec:llc:params}
This subsection is largely an expository account
of Sections~5.2 and~6.1 of Kaletha's article~\cite{kaletha19a}.
Our goal is to describe the regular 
supercuspidal $L$-parameters
and their construction from tame elliptic extra regular pairs.

The definition of regularity for $L$-parameters
\cite[Definition 5.2.3]{kaletha19a}
is not important for us,
so we omit the precise statement:
roughly speaking, a parameter~$\varphi$ is regular
if it takes the wild subgroup to a torus
and the centralizer of the inertia subgroup
is abelian
(in fact, this is the definition
of a ``strongly regular'' parameter).
Consequently, the groups $S_\varphi$ 
and $S_\varphi^\natural$ appearing
in the statement of the formal degree conjecture are abelian,
so that their irreducible representations are one-dimensional.
This means that we can ignore the factor $\dim\rho$
appearing in the Galois side of the formal degree conjecture,
since it equals~$1$.

However, we should explain the relationship
between regularity and torus-character pairs.
To classify regular parameters,
Kaletha introduced an auxiliary category
of \emph{regular supercuspidal $L$-packet data}
whose objects are quadruples $(S,\hat\jmath,\chi,\theta)$
consisting of
\begin{itemize}
\item
a tame torus~$S$ of dimension the absolute rank of~$G$,
\item
an embedding $\hat\jmath:\widehat S\to\widehat G$
of complex reductive groups whose $\widehat G$-conjugacy class
is Galois stable,
\item
a minimally ramified set of $\chi$-data for~$\RS(G,S)$, and
\item
a character $\theta:S(k)\to\C^\times$.
\end{itemize}
For the meaning of
\emph{minimally ramified $\chi$-data},
see \Cref{thm66,thm67} in \Cref{sec:gal:chi}.
These objects are required to satisfy additional
conditions \cite[Definition 5.2.4]{kaletha19a}
that do not concern us here.
One can also define a morphism of such data,
organizing them into a category
in which all morphisms are isomorphisms.

\begin{warning}
In the definition of regular supercuspidal $L$-packet data,
we do not assume that $S$ is a maximal torus of~$G$.
In general, the $\widehat G$-conjugacy class of~$\hat\jmath$
gives rise to a Galois-stable
$G(\bar k)$-conjugacy class of embeddings
$S_{\bar k}\into G_{\bar k}$
whose elements are called \emph{admissible}
(with respect to $\hat\jmath$)
\cite[Section~5.1]{kaletha19a}.
Since $S$ is elliptic,
this $G(\bar k)$-conjugacy class contains
embeddings defined over~$k$.
However, there is no canonical such embedding,
or even $G(k)$-conjugacy class of embeddings.
This failure is related to
the need to organize supercuspidal
representations into $L$-packets.
\end{warning}

The key property of the category 
of regular supercuspidal $L$-packet data
is that the isomorphism classes of its objects
are in natural bijection with
equivalence classes of regular supercuspidal parameters.
Given a regular supercuspidal $L$-packet
datum~$(S,\hat\jmath,\theta,\chi)$,
its parameter is the composition
\[
{}^L j_\chi\circ{}^L \theta,
\]
where ${}^L j_\chi$
is the $L$-embedding of \Cref{sec:llc:emb};
this is the direction of the correspondence
that we need to understand
when we compute, in \Cref{sec:gal},
the absolute value
of the adjoint $\gamma$-factor.

Let $(S,\theta)$ be a tame elliptic extra regular pair.
Assume that there is at least one
admissible embedding~$j$ of~$S$ as
a maximal torus of~$G$.
Instead of just pulling back
the character~$\theta$ to~$jS$,
we need to modify it slightly:
define
\[
j\theta' \defeq \theta\circ j^{-1}\cdot\varepsilon
\]
where $\varepsilon =\varepsilon_{f,\tn{ram}}
\cdot\varepsilon^\tn{ram}$
is a certain tamely-ramified
Weyl-invariant character of~$jS$
\cite[Section~5.3]{kaletha19a}.
Kaletha used the character formula of
Adler, DeBacker, and Spice
\cite{adler_spice08,adler_spice09,debacker_spice18}
to construct from $(S,\theta,j)$
a certain minimally ramified set~$\chi$ of $\chi$-data,
which appears, for one, in the definition of~$\varepsilon$.
Then the $L$-packet corresponding
to $(S,\hat\jmath,\chi,\theta)$ consists
of the set of regular supercuspidal representations
\[
\pi_{(jS,j\theta')}
\]
where $j:S\into G$ ranges over
the $G(k)$-conjugacy classes of admissible embeddings.

\section{Automorphic side}
\label{sec:aut}
In \Cref{sec:llc:reps}, we outlined Yu's construction
of supercuspidal representations.
In this section we calculate
the formal degree of such a representation.
This result is of independent interest,
and could be used to verify the formal degree
conjecture for broader classes
of supercuspidal representations than those
considered in this paper.

The basic idea of the computation is quite simple,
but various technical complications arise
in the process.
As \Cref{sec:aut:yu} explains,
Yu's representations are obtained by compact-induction
of a finite-dimensional representation
of a compact-open (or really, compact-mod-center)
subgroup of~$G(k)$.
There is a general formula for the formal degree
of such a representation in terms
of the dimension of the starting representation
and the volume of the subgroup.
\Cref{sec:aut:deg} explains this formula
and reviews the notion of formal degree.
To compute the formal degree of a Yu representation,
then, one need only compute two numbers,
a dimension and a volume.

The dimension comes from Deligne-Lusztig theory
and is straightforward to compute in our case.
We work it out in \Cref{sec:aut:reg},
where we specialize the formal degree computation
to the case of a regular supercuspidal representation.

The volume comes from Bruhat-Tits theory,
and is much more difficult to compute.
Still, the basic idea is clear.
Computing the volume of a compact-open subgroup
amounts to computing its index
in a larger group of known volume,
so the volume computation boils
down to an index computation.
Using the Moy-Prasad isomorphism,
that index computation, in turn,
boils down to a computation
of the subquotients in
the Moy-Prasad filtration on the Lie algebra.

The groups used in Yu's construction
generalize the subgroups 
of the Moy-Prasad filtration.
In \Cref{sec:aut:split,sec:aut:tame,sec:aut:yu_grp}
we review their
construction and explain various ways
in which Yu's theory is an elaboration
of the theory of Moy and Prasad
\cite{moy_prasad94,moy_prasad96},
or going back even farther,
the theory of Bruhat and Tits
\cite{bruhat_tits72,bruhat_tits84},
Our goal in that lengthy section is to generalize
the Moy-Prasad isomorphism to Yu's groups,
and to understand the extent to which the Lie algebra
of one of Yu's groups
decomposes as a direct sum of root lines.
After these preliminaries,
we compute the dimension of such a Lie algebra
in \Cref{sec:aut:len}.
The Moy-Prasad isomorphism then translates
this dimension into the subgroup index
that we need to compute.
At this point the main steps are in place,
and in \Cref{sec:aut:comp} we walk up
the staircase and finish the computation.

To make statements like \Cref{thm30} easier to read,
in \Cref{sec:aut:not} we introduce
notation particular to \Cref{sec:aut}
for reductive groups and their topological groups
of rational points.

\subsection{Notation}
\label{sec:aut:not}
It is important for many reasons 
to distinguish between a linear algebraic $k$-group
and the group of its rational points.
The former is a $k$-scheme
and the latter is just an abstract group,
a topological group if $k$ has a topology.
Many constructions are easier on the level of schemes,
but for the kind of representation theory
we consider in this paper,
one can work only with the group of rational points.

It is conventional in algebraic geometry
to denote a $k$-group by~$G$
and the group of its rational points by~$G(k)$.
Following the convention in this section,
however, would create a confusing proliferation
of ``$(k)$'' suffixes.
In \Cref{sec:aut} only, therefore, we underline,
denoting $k$-groups by $\ul G, \ul H,\dots$
and their groups of rational points by $G,H,\dots$:
\[
G \defeq \ul G(k).
\]
In particular, we write
$\vec{\ul G}$ for a twisted Levi sequence
and $\vec G$ for the sequence of topological groups
obtained from it by taking $k$-points.
The convention extends to $\cal O$- and~$\kappa$-schemes as well.
For instance, we write $\ul G_{x,r}$
for the smooth $\cal O$-group,
constructed by Yu~\cite[Section~8]{yu02},
whose group of $\cal O$-points is the
Moy-Prasad group~$G_{x,r}$.
Similarly, $\ul G_{x,0:0+}$
denotes the maximal reductive quotient
of the special fiber of~$\ul G_x$,
a $\kappa$-group.

\subsection{Formal degree}
\label{sec:aut:deg}
Here we review the formal degree,
following the relevant section 
of Renard's monograph on representations
of $p$-adic groups \cite[IV.3]{renard10},
then calculate the formal degree
of a compactly induced representation.
Contrary to the conventions of the rest of the paper,
in this subsection only,
let $G$ be a unimodular locally profinite group,
let $Z$ be the center of~$G$,
and let $(\pi,V)$ be a smooth irreducible
representation of~$G$.
Let $A$ be a closed subgroup of~$Z$
such that $Z/A$ is compact,
and let $\mu$ be a Haar measure on~$G/A$.
Several of our results use $A$
in the statement but are independent
of the choice of~$A$.

The \emph{matrix coefficient} of~$\pi$
with respect to~$v\in V$ and $v^\vee\in V^\vee$
(where $V^\vee$ is the smooth dual of~$V$)
is the function $\pi_{v,v^\vee}:G\to\C$ defined by
$\pi_{v,v^\vee}(x) = \langle\pi(x)v,v^\vee\rangle$.
Since $\pi$ is irreducible,
there is a character~$\chi$ of~$Z$,
called the \emph{central character} of~$\pi$,
such that $\pi(z) = \chi(z)$ for all $z\in Z$.
Assume that the central character is unitary.
In this case, the function
$x\mapsto|\pi_{v,v^\vee}(x)|^2$
is constant on cosets of~$A$,
and hence defines a function on~$G/A$.
We write $|\pi_{v,v^\vee}|_{L^2(G/A,\mu)}^2$
for the integral of this function,
and say that $(\pi,V)$ is \emph{discrete series}
(with respect to~$A$) if
$|\pi_{v,v^\vee}|_{L^2(G/A,\mu)}<\infty$
for all~$v\in V$ and $v^\vee\in V^\vee$.
This condition is independent of~$A$,
and also of the choice of Haar measure on~$G/A$.

In practice, it is useful to slightly weaken the definition
of a discrete series representation,
and to define a representation to be
\emph{essentially discrete series}
if it becomes discrete series after twisting
by some character of the group.

It can be shown that every discrete series
representation is \emph{unitary},
not in the sense of Hilbert space representations,
but in the sense that it admits
a positive-definite $G$-invariant Hermitian product.
The resulting isomorphism between
$\overline V$ and~$V^\vee$
defines a matrix coefficient~$\pi_{v,w}$
for $v,w\in V$.
Set $\pi_v \defeq \pi_{v,v}$.

For a discrete series representation~$(\pi,V)$,
one would hope for a relationship
between the norm of a vector
and the $L^2$-norm of its matrix coefficient.
Although these norms are not equal in general,
it turns out that they differ
by a multiplicative constant
depending only on~$\pi$ and the Haar measure on~$G/A$,
not on the vector.
This constant is called the \emph{formal degree}.

\begin{definition}[{\cite[IV.3.3]{renard10}}]
Let $(\pi,V)$ be an essentially discrete series representation
and let $\mu$ be a Haar measure on~$G/A$.
If $\pi$ is in addition discrete series
then there exists a positive real
constant~$\deg(\pi,\mu)$,
called the \emph{formal degree} of~$\pi$,
such that for all $v\in V$,
\[
|\pi_v|^2_{L^2(G/A,\mu)} = \frac{|v|^2}{\deg(\pi,\mu)}.
\]
In general, we define the formal degree of~$\pi$
as the formal degree of any discrete series
representation of~$G$ obtained from~$\pi$
by twisting by a character of~$G$.
\end{definition}

Evidently the formal degree scales inversely
with the Haar measure used to define it:
\[
\deg(\pi,c\mu) = c^{-1}\deg(\pi,\mu).
\]

\begin{remark}
\label{thm49}
Assume $Z$ is compact in this remark for simplicity.
In harmonic analysis, one studies
the set of irreducible unitary representations of~$G$
by endowing it with a certain topology and measure,
called the \emph{Plancherel measure}.
The measure, though not the topology, depends
on a choice of Haar measure~$\mu$ on~$G$.
An irreducible unitary representation $(\pi,V)$
is discrete series if and only if
it is an isolated point of positive measure,
and that measure is the formal degree of~$\pi$
with respect to~$\mu$.
In particular, if $G$ is compact
and we choose the Haar measure on~$G$
giving it volume one
then the formal degree of~$(\pi,V)$
is just the dimension of~$V$.
\end{remark}

Discrete series representations are closely related
to supercuspidal representations,
but neither notion implies the other.
Recall that a smooth irreducible representation
is \emph{supercuspidal} if its matrix coefficients
have compact support modulo the center,
or equivalently, modulo~$A$.
If a supercuspidal representation
has unitary central character
then it is certainly discrete series,
since compactly supported functions are square-integrable;
but since not every compactly supported function
is square-integrable,
in general, there are discrete series representations
that are not supercuspidal.
At the same time, since supercuspidal representations
need not have unitary central character,
not all supercuspidal representations are discrete series.
However, when $G$ is a $p$-adic reductive group,
every supercuspidal representation of~$G$
is essentially discrete series,
and we may thus speak of the representation's formal degree.

Since Yu's supercuspidal representations
are compactly induced,
we compute their formal degree using
a general formula for the formal degree
of a compactly induced representation.
As a preliminary step, we define
a natural Hermitian product
on a compactly induced representation.

Let $K$ be an open, compact-mod-$A$ subgroup of~$G$
(this condition is independent of~$A$),
let $(\rho,W)$ be a smooth irreducible unitary
representation of~$K$,
and let $(\pi,V)$ be the representation of~$G$
compactly induced from~$(\rho,W)$.
That is, $V$ is the space of
smooth functions $f:G\to W$
whose support is compact-mod-$K$
and that satisfy $f(hx) = \rho(h)f(x)$
for all $h\in K$ and~$x\in G$.
The representation $(\pi,V)$ is unitary;
in fact, an invariant scalar product
is given by the formula
\begin{equation}
\label{eq19}
\langle f_1,f_2\rangle
= \int_{K\backslash G} \langle f_1(x),f_2(x)\rangle
\dif\mu_{K\backslash G}(x),
\end{equation}
where $\mu_{K\backslash G}$ is any 
positive $G$-invariant Radon measure on~$K\backslash G$,
for instance, the counting measure.

\begin{lemma}
\label{thm20}
Let $(\rho,W)$ be a finite-dimensional unitary representation of~$K$,
let $(\pi,V)$ be the compact induction of~$W$ to~$G$,
and let $\mu$ be a Haar measure on~$G/A$.
If $\pi$ is irreducible then
\[
\deg(\pi,\mu) = \frac{\dim\rho}{\vol(K/(K\cap A),\mu)}.
\]
\end{lemma}

\begin{proof}
We start by defining an isometric embedding $W\into V$.
Given a vector $w\in W$, define $\dot w\in V$ by
\[
\dot w(x) = \one_H(x)\rho(x) w.
\]
The space $K\backslash G$ is discrete because $K$ is open,
so we can take the measure~$\mu_{K\backslash G}$
in \cref{eq19} to be the counting measure.
With this choice, the map $W\to V$ defined by $w\mapsto\dot w$
is an isometric embedding.
It follows that the matrix coefficient of~$\dot w$
is the extension by zero of the matrix coefficient of~$w$,
and that their $L^2$-norms coincide provided that
we take the Haar measure on $K/(K\cap A)$
to be the restriction of~$\mu$,
denoted also by~$\mu$.
Then for any nonzero $w\in W$,
\[
\deg(\pi,\mu) = |\dot w|^2\cdot|\pi_{\dot w}|_{L^2(G/A,\mu)}^{-2}
= |w|^2\cdot|\rho_w|_{L^2(K/(K\cap A),\mu)}^{-2}
= \deg(\rho,\mu).
\]
Finally, since the formal degree scales
inversely to the Haar measure used to define it,
\[
\deg(\rho,\mu) = \vol(K/(K\cap A),\mu_0)^{-1}
\deg(\rho,\mu_0)
\]
where $\mu_0$ is the measure on~$K/(K\cap A)$
assigning it total volume one.
The degree of~$\rho$ with respect to this measure
is just the dimension of~$W$.
\end{proof}

\subsection{Concave-function subgroups: split case}
\label{sec:aut:split}
Suppose $\ul G$ is split
with split maximal torus~$\ul T$
and root system $\RS=\RS(\ul G,\ul T)$,
so that $\RS=\underline\RS$.
In this setting, Bruhat and Tits
showed~\cite[Section~6.4]{bruhat_tits72}
how to construct from a function
$f:\RS\cup\{0\}\to\widetilde\R$
and point $x\in\cal A(\ul G,\ul T)$
a subgroup~$G_{x,f}$ of~$G$
and a subgroup~$\frak g_{x,f}$ in~$\frak g$.
In this subsection we review Bruhat
and Tits's construction.
The eventual goal, in later subsections,
is to explain how their construction generalizes
to the construction of subgroups
that appear in Yu's construction of supercuspidals,
and to then study Yu's subgroups.

The definitions of~$G_{x,f}$ and~$\frak g_{x,f}$
are quite natural.
The point~$x$ defines, or ``is'',
depending on one's point of view,
a family of additive valuations
$(v^\alpha_x:\ul U^\alpha(k)\to\R)_{\alpha\in\RS}$,
where $\ul U^\alpha$ is the root group of~$\alpha$.
Since $\ul U^\alpha$ is canonically isomorphic
to the root line~$\frak g^\alpha$,
we may also think of~$v^\alpha_x$ as a valuation
$\frak g^\alpha\to\R$.
Now let
\[
U^\alpha_{x,r} \defeq \{u\in U^\alpha
: v_x^\alpha(u)\geq r\},
\qquad
\frak g^\alpha_{x,r} \defeq \{X\in \frak g^\alpha
: v_x^\alpha(X)\geq r\}.
\]
As for the point $\alpha=0$,
we can think of~$\ul T$ as the root group~$\ul U^0$
and its Lie algebra~$\frak t$ as the root space~$\frak g^0$.
These objects carry their own filtrations:
let
\[
T_r \defeq \{t\in T : \forall \chi\in X^*(T),\;
\ord(\chi(t)-1)\geq r\}
\]
and let
\[
\frak t_r \defeq \{X\in T : \forall \chi\in X^*(T),\;
\ord(\dif\chi(X))\geq r\}.
\]
The objects~$T_r$ and~$\frak t_r$
do not depend on~$x$,
but we reserve the right to denote them by $T_{x,r}$
and~$\frak t_{x,r}$ for uniformity of notation.

\begin{warning}
The group~$G^0$ written here is unrelated
to the zeroth group in a twisted Levi sequence,
even though the notation for the two is the same.
The two notations never appear in the same subsection,
however, so there is little risk of confusion.
\end{warning}

Given a function $f:\RS\cup\{0\}\to\widetilde\R$,
let
\[
U^\alpha_{x,f} \defeq U^\alpha_{x,f(\alpha)},
\qquad
\frak g^\alpha_{x,f} \defeq \frak g^\alpha_{x,f(\alpha)}
\]
for any $\alpha\in\RS\cup\{0\}$.
The group $G_{x,f}$ is then defined as
the subgroup of~$G$ generated by
the subgroups $U^\alpha_{x,f}$
with $\alpha\in\RS\cup\{0\}$,
and the lattice $\frak g_{x,f}$ is defined as
the subgroup of~$\frak g$ spanned by
the subgroups $\frak g^\alpha_{x,f}$
with $\alpha\in\RS\cup\{0\}$.

\begin{remark}
\label{thm7}
When $r=\infty$ the group $U^\alpha_{x,\infty}$ is trivial,
and when in addition $\alpha\neq0$
we can recover the filtrations
on the root groups and root lines
as $\frak g^\alpha_{x,r} = \frak g_{x,f}$
and $U^\alpha_{x,r} = G_{x,f}$ where
\[
f(\beta) = \begin{cases}
r      & \tn{if } \beta=\alpha \\
\infty & \tn{if not.}
\end{cases}
\]
\end{remark}

In order for the construction of~$G_{x,f}$ to behave nicely
we must assume that $f$ is nonnegative
and \emph{concave}, that is, that
for all finite families $(\alpha_i)_{i\in I}$
of elements of~$\RS\cup\{0\}$,
\[
f\Bigl(\sum_{i\in I} \alpha_i\Bigr)
\leq \sum_{i\in I}f(\alpha_i)
\]
whenever $\sum_{i\in I} \alpha_i\in\RS\cup\{0\}$.
We can define~$\frak g_{x,f}$ for any~$f$ whatsoever,
but when $f$ is concave,
$\frak g_{x,f}$ is 
a sub Lie algebra of~$\frak g$.

This completes our discussion of the split case.
We next generalize the split case to the tame case,
a simple exercise in Galois descent.

\begin{remark}
For simplicity of narrative we attributed
the construction of~$G_{x,f}$ to Bruhat and Tits,
but Yu is partially responsible
for the construction, even in the split case.
Bruhat and Tits worked only with functions~$f$
such that $f(0)=0$,
but Yu extended their theory to all~$f$.
\end{remark}

\subsection{Concave-function subgroups: tame case}
\label{sec:aut:tame}.
We no longer assume that $\ul G$ is split,
only that it split over a tamely ramified extension.
Hence we must distinguish between
$\ul R\defeq\ul R(\ul G,\ul T)$
and~$R\defeq R(\ul G,\ul T)$.
In this setting, Yu showed how to generalize
the constructions of the previous subsection,
he constructed for each function
$f:\RS\cup\{0\}\to\widetilde\R$
and point $x\in\cal A(\ul G,\ul T)$,
a subgroup $G_{x,f}$.
In fact, Yu defined his construction
only for a special class of functions,
which we review in the next subsection.
However, for aesthetic reasons,
we found it helpful to review the theory
in this moderately greater generality.

Let $\ell\supseteq k$ be some fixed
tame Galois extension of~$k$
splitting the maximal torus~$\ul T$
and let $\Gamma\!_{\ell/k}$ be the Galois
group of~$\ell$ over~$k$.
A function $f:\underline\RS\cup\{0\}\to\widetilde\R$
can be interpreted as a Galois-invariant function
$\RS\cup\{0\}\to\widetilde\R$,
and we say that $f$ is \emph{concave}
if the associated Galois-invariant function is concave.

Since $f$ is Galois-invariant,
the subgroups~$(\frak g_{\ell})_{x,f}$
and~$\ul G(\ell)_{x,f}$ are Galois-invariant
and we define
\[
\frak g_{x,f} \defeq (\frak g_{\ell})_{x,f}^{\Gamma\!_{\ell/k}},
\qquad
G_{x,f} \defeq \ul G(\ell)_{x,f}^{\Gamma\!_{\ell/k}}.
\]
These groups do not depend on the choice of~$\ell$.
We need not assume that $f$ is concave
to define these objects,
although they are best-behaved in that case.
We can also construct for each~$\ul\alpha$
in $\underline\RS\cup\{0\}$ a root space
\[
\frak g^{\ul\alpha}\defeq\Bigl(\bigoplus_{\alpha\in\alpha(\bar k)}
\frak g_\ell^\alpha\Bigr)^{\Gamma\!_{\ell/k}}.
\]
When $\alpha=0$, the root space $\frak g^0$
is just the Lie algebra of~$\ul T$.
Combining this construction with \cref{thm7},
we see that each root space~$\frak g^{\ul\alpha}$
admits a natural filtration: define
\[
\frak g^{\ul\alpha}_{x,f}=
\Bigl(\bigoplus_{\alpha\in\ul\alpha(\bar k)}
(\frak g_\ell^\alpha)_{x,f}\Bigr)^{\Gamma\!_{\ell/k}}.
\]
It follows immediately from the definitions that
\[
\frak g_{x,f} = \bigoplus_{\ul\alpha\in\underline\RS\cup\{0\}}
\frak g^{\ul\alpha}_{x,f}.
\]
More generally, given a 
subset~$\underline\RS'\subseteq\underline\RS\cup\{0\}$,
let
\[
\frak g^{\underline\RS'} \defeq
\bigoplus_{\ul\alpha\in\underline\RS'} \frak g^{\ul\alpha},
\qquad
\frak g^{\underline\RS'}_{x,f} \defeq
\bigoplus_{\ul\alpha\in\underline\RS'} \frak g^{\ul\alpha}_{x,f}.
\]

\begin{warning}
In comparison to the function of \Cref{thm7},
the function $f:\underline\RS\cup\{0\}\to\widetilde\R$
defined, for a fixed $\ul\alpha\in\underline\RS$, by
\[
f(\ul\beta) = \begin{cases}
r      & \tn{if } \ul\beta=\ul\alpha \\
\infty & \tn{if not}
\end{cases}
\]
is generally not concave.
Failure of concavity relates
to the fact that when $\ul T$ is not split,
there is generally no ``root group''
(or even ``root variety'')~$\ul U^{\ul\alpha}$
whose Lie algebra is~$\frak g^{\ul\alpha}$.
\end{warning}

\begin{remark}
If we assume that $f$ is finite,
or without loss of generality,
that $f$ takes values in~$\R$,
then the group $G_{x,f}$ can be interpreted
as the integral points of an $\cal O$-group:
that is, there is a canonical $\cal O$-group
$\ul G_{x,f}$, an integral model of~$\ul G$,
such that
\[
\ul G_{x,f}(\cal O) = G_{x,f}
\]
as subgroups of~$G$.
The construction is due to Yu
\cite[Section~8]{yu02}.
\end{remark}

Having defined the group~$G_{x,f}$
and Lie algebra~$\frak g_{x,f}$,
we now study them.
There are three areas of interest
for our later applications.

First, in our calculation of the formal degree,
it greatly simplifies notation to reduce
to the case where $G$ has anisotropic center.
We record here the lemma effecting this reduction.

\begin{lemma}
\label{thm9}
Let $\ul G$ be a reductive $k$-group,
let $\ul T$ be a tame maximal torus of~$\ul G$,
let $f:\underline\RS(\ul G,\ul T)\to\widetilde\R$
be a positive concave function,
let $\ul A\subseteq\ul T$ be a split central torus
of~$\ul G$,
and let $y\in\cal A(\ul T,\ul G)$.
Then the groups $G_{y,f}/(A\cap G_{y,f})$
and $(G/A)_{y,f}$ are identical subgroups of~$G/A$.
\end{lemma}

\begin{proof}
This follows in the split case using Hilbert's
Theorem~90, and the general case follows
immediately from the split case by taking
Galois invariants.
\end{proof}

Second, when we compute certain subgroup indices
in \Cref{sec:aut:comp},
we need to understand how to intersect groups
of the form $G_{x,f}$, for fixed~$x$.

\begin{lemma}
\label{thm48}
Let $f,g:\ul R\cup\{0\}\to\widetilde\R$
be positive concave functions.
Then
\[
G_{x,f}\cap G_{x,g} = G_{x,\max(f,g)}.
\]
\end{lemma}

\begin{proof}
In the case where $G$ is split,
a classical result of Bruhat and Tits
\cite[(6.4.48)]{bruhat_tits72}
can be used to show \cite[8.3.1]{yu02}
that the natural multiplication map
\[
\prod_{\alpha\in R\cup\{0\}}
U^\alpha_{x,f(\alpha)} \to G_{x,f}
\]
is a bijection for a certain ordering of the factors.
Using tame descent, this observation
reduces the proof to the obvious fact 
(still in the split case) that
$U^\alpha_{x,r}\cap U^\alpha_{x,s}
= U^\alpha_{x,\max(r,s)}$.
\end{proof}

Third and lastly,
we compare subgroup indices between $G$ and~$\frak g$,
generalizing the Moy-Prasad isomorphism.
In \Cref{sec:aut:comp},
this comparison reduces a volume computation
to a length computation,
which we carry out in \Cref{sec:aut:len}.

\begin{lemma}
\label{thm8}
Let $f,g:\underline\RS\cup\{0\}\to\widetilde\R$
be positive concave functions such that $f\leq g$.
Assume in addition that the following condition
is satisfied:
\[
g(a) \leq \sum_{i\in I} f(a_i) + \sum_{j\in J} f(b_j)
\]
for all non-empty finite sequences
$(a_i)_{i\in I}$ and~$(b_j)_{j\in J}$
of elements of $\underline\RS\cup\{0\}$
such that $a\defeq\sum_{i\in I}a_i + \sum_{j\in J}b_j
\in\underline\RS\cup\{0\}$.
Then
\begin{enumerate}
\item
$[G_{x,f},G_{x,f}]\subseteq G_{x,g}$,
so that the group $G_{x,f:g}$ is abelian, and
\item
there is a canonical isomorphism
$\frak g_{x,f:g}\simeq G_{x,f:g}$ of abelian groups.
\end{enumerate}
\end{lemma}

\begin{proof}
This follows from~\cite[6.4.44]{bruhat_tits72}
and~\cite[6.4.48]{bruhat_tits72}
in the split case
and \cite[Section~2]{yu02} in general.
\end{proof}

\begin{remark}
In \Cref{thm8},
it is tempting to instead impose the simpler condition
that $g(a) \leq \sum_{i\in I} f(a_i)$
for all non-empty finite sequences
$(a_i)_{i\in I}$ such that $\sum_{i\in I}a_i = a$.
However, this stronger condition would
significantly weaken the conclusion:
the condition implies, by taking the constant sequence,
that $g\leq f$, so that $g=f$.
\end{remark}

\begin{corollary}
\label{thm56}
Let $f,g:\underline\RS\cup\{0\}\to\R$
be positive concave functions such that $f\leq g$,
and suppose there is a chain of concave functions
\[
f = f_0 \leq f_1 \leq \cdots \leq f_n = g
\]
such that for each~$i$ with $1\leq i\leq n$,
the pair $(f_{i-1},f_i)$ satisfies
the conditions of \cref{thm8}.
Then
\[
|G_{x,f:g}| = |\frak g_{x,f:g}|.
\]
\end{corollary}

\subsection{Yu's groups}
\label{sec:aut:yu_grp},
In our application,
it is enough to work with the subgroups
constructed from a certain restricted
class of concave functions,
those constructed from admissible sequences
and tame twisted Levi sequences.
The construction specializes
that of \Cref{sec:aut:tame},
so we work in the same setting.
After reviewing Yu's construction
of these groups,
we specialize the theory
of \Cref{sec:aut:tame}
to show that they admit
a Moy-Prasad isomorphism.

A sequence $\vec r = (r_i)_{i=0}^d$ in~$\widetilde\R$
is \emph{admissible} if there is some~$j$
with $0\leq j\leq d$ such that
$0\leq r_0 = \cdots = r_j$ and
$\tfrac12 r_j \leq r_{j+1}\leq\cdots r_d$.
The admissible sequence is \emph{weakly increasing}
if, in addition, $r_i \leq r_{i+1}$
for all~$i$ with~$0\leq i\leq d-1$.

Recall the definition of
a (tame) twisted Levi sequence~$\vec{\ul G}$
from \Cref{sec:llc:reps}.
We assume that the tame maximal torus~$\ul T$
of~$\ul G$ is contained in~$\ul G^0$, so that
$x\in\cal B(\ul G^i)$ for each~$i$.
For each $i$ with $0\leq i\leq d$,
let $\underline\RS_i\defeq\underline\RS(\ul G^i,\ul T)$.

It is sometimes necessary to work
with truncated twisted Levi sequences.
Given integers $a$ and~$b$ such that $0\leq a\leq b\leq d$,
let
\[
\vec{\ul G}^{[a,b]}\defeq
(\ul G^a\subsetneq\cdots\subsetneq\ul G^b);
\]
given an integer~$i$ with $0\leq i\leq d$, define
$\vec{\ul G}^{(i)} \defeq \vec{\ul G}^{[0,i]}$.

Given an admissible sequence~$\vec r$
and a twisted Levi sequence $\vec{\ul G}$
of the same length~$d$,
define the function $f_{\vec r}:\underline\RS\to\widetilde\R$ by
\[
f_{\vec r}(\ul\alpha) \defeq \begin{cases}
r_0 & \tn{if } \ul\alpha\in\ul\RS_0\cup\{0\} \\
r_i & \tn{if } \ul\alpha\in\ul\RS_i\setminus\ul\RS_{i-1},\, 1\leq i\leq d.
\end{cases}
\]
Since $\vec r$ is admissible,
the function $f_{\vec r}$ is
concave~\cite[Lemma~1.2]{yu01}.
Hence we may define
\[
\vec G_{x,\vec r}
\defeq G_{x,f_{\vec r}},
\qquad
\vec{\frak g}_{x,\vec r}
\defeq \frak g_{x,f_{\vec r}}.
\]
Given a second admissible sequence~$\vec s$
of length~$d$ with $r_i\leq s_i$ for all~$i$,
define
\[
\vec G_{x,\vec r:\vec s}
\defeq G_{x,f_{\vec r}:f_{\vec s}},
\qquad
\vec{\frak g}_{x,\vec r:\vec s}
\defeq \frak g_{x,f_{\vec r}:f_{\vec s}}.
\]

\begin{remark}
The group $\vec G_{x,\vec r}$
depends only on $\vec{\ul G}$,
$x$, and~$\vec r$;
in particular, it is independent
of the choice of torus~$T$,
provided that $x\in\cal B(T)$
\cite[Section~1]{yu01}.
The same cannot necessarily be said
for a general group of the form $G_{x,f}$,
however, because the domain
of definition of~$f$ knows something about
the torus~$T$, namely,
the Galois action on the root system.
This Galois action can vary among maximal tori whose
buildings are identical subsets of~$\cal B(G)$.
\end{remark}

In this setting, Yu generalized
\cite[Lemma~1.3 and Corollary~2.4]{yu01}
the Moy-Prasad isomorphism.

\begin{lemma}
\label{thm10}
Let $\vec{\ul G}$ be a tame Levi sequence of length~$d$
and let $\vec r$ and~$\vec s$ be admissible sequences
of length~$d$ such that for all~$i$,
\begin{equation}
\label[condition]{eq14}
0 < r_i \leq s_i \leq
\min(r_i,\dots,r_d) + \min(\vec r).
\end{equation}
Then $\vec G_{x,\vec r:\vec s}$
is an abelian group canonically isomorphic
to $\vec{\frak g}_{x,\vec r:\vec s}$.
\end{lemma}

\begin{proof}
\cref{eq14} implies that the pair
$(f_{\vec r},f_{\vec s})$ satisfies
the conditions of \cref{thm8}.
\end{proof}

\begin{corollary}
\label{thm57}
Let $\vec{\ul G}$ be a tame Levi sequence
of length~$d$ and
let $\vec r$ and~$\vec s$ be weakly increasing 
admissible sequences of length~$d$ such that
$0 < r_i \leq s_i < \infty$ for all~$i$.
Then
\[
|\vec G_{x,\vec r:\vec s}|
= |\vec{\frak g}_{x,\vec r:\vec s}|.
\]
\end{corollary}

\begin{proof}
Since $\vec r$ is weakly increasing,
\cref{eq14} simplifies
to the condition that
\begin{equation}
s_i\in[r_i,r_i+r_0] \quad \tn{ for all~$i$.}
\end{equation}
It is now an elementary but tedious exercise
to construct a chain $(\vec s\,^{(j)})_{0\leq j\leq N}$
of weakly increasing admissible sequences
$\vec s\,^{(j)}=(s^{(j)}_0\leq\cdots\leq s^{(j)}_d)$,
where $N\gg0$, such that
$s^{(0)}=\vec r$ and $s^{(d)}=\vec s$,
such that $s^{(j)}_i \leq s^{(j+1)}_i$
for all $i$ and~$j$ with $0\leq j\leq N-1$,
and such that each pair
$(\vec s\,^{(j-1)},\vec s\,^{(j)})$
satisfies \cref{eq14} of \Cref{thm10}.
After completing this exercise,
we invoke \Cref{thm56}.
\end{proof}

We need the \lcnamecref{thm57} in the following special case only.

\begin{corollary}
\label{thm15}
Let $\vec{\ul G}$ be a tame Levi sequence
of length~$d$ and let $\vec r$ be a weakly
increasing admissible sequence of length~$d$.
Then
\[
|\vec G_{x,0+:\vec r}| = |\vec{\frak g}_{x,0+:\vec r}|.
\]

\end{corollary}

\subsection{Length computation}
\label{sec:aut:len}
Retaining the notation of \Cref{sec:aut:tame},
let $f:\underline\RS\to\R$ be a positive function
and let $\ul R'\subseteq\ul R$ be a subset.
Our goal in this subsection
is to compute the length of the $\cal O$-module
$\frak g^{\ul R'}_{x,0+:f}$,
culminating in \Cref{thm18}.

We start by studying the jumps
in the filtration on $\frak g^{\ul\alpha}_x$.
For $\ul\alpha\in\underline\RS$,
consider the set
\[
\ord_x\ul\alpha \defeq \{t\in\R :
\frak g^{\ul\alpha}_{x,t:t+}\neq0\}
\]
of jumps in the Moy-Prasad filtration
of~$\frak g^{\ul\alpha}$,
defined and studied by DeBacker and Spice
\cite[Definition~3.6]{debacker_spice18}.
A full description of~$\ord_x\ul\alpha$
requires an understanding of the point~$x$,
and thus the way in which $\cal B(\ul T)$
embeds in $\cal B(\ul G)$.
This is quite difficult in general.
But for us it is enough to know
several weak properties of these sets.

\begin{lemma}
\label{thm16}
Let $\ul T$ be a tame maximal torus of~$\ul G$,
let $x\in\cal A(\ul G,\ul T)$,
and let $\ul\alpha\in\underline\RS$.
Then
\begin{enumerate}
\item
$\ord_x(-\ul\alpha)=-\ord_x\ul\alpha$,
\item
$\ord_x\ul\alpha$ is an $\ord_k(k_{\ul\alpha}^\times)$-torsor, and
\item
$\len \frak g^{\ul\alpha}_{x,t:t+}
= [\kappa_{\ul\alpha}:\kappa]\one_{\ord_x\ul\alpha}(t)$.
\end{enumerate}
\end{lemma}

\begin{proof}
Property~(1) follows from a $\PGL_2$
calculation~\cite[Corollary~3.11]{debacker_spice18}.
As for the other properties,
earlier we mentioned that $\frak g^{\ul\alpha}$
was isomorphic to~$k_{\ul\alpha}$,
though not canonically.
Choose one such isomorphism $\phi:k_{\ul\alpha}
\to\frak g_{\ul\alpha}$.
This isomorphism is compatible
with the Moy-Prasad filtration in the following sense:
there is a real number~$r_0$ such that for all $r\in\R$,
the isomorphism~$\phi$ restricts to an isomorphism
$k_{\ul\alpha,r+r_0}\simeq\frak g_{\ul\alpha,r}$
of $\cal O$-modules.
Properties~(2) and~(3) now follow immediately.
\end{proof}

\begin{corollary}
\label{thm59}
Let $\ul T$ be a tame maximal torus of~$\ul G$,
let $x\in\cal A(\ul G,\ul T)$,
let $f:\underline\RS\cup\{0\}\to\R$ be a positive function,
and let $\ul R'\subseteq\ul R$.
Then
\[
\len\frak g^{\ul R'}_{x,0+:f}
= \sum_{\ul\alpha\in\underline R'}
\sum_{0<t<f(\ul\alpha)} [\kappa_{\ul\alpha}:\kappa]
\one_{\ord_x\ul\alpha}(t)
\]
\end{corollary}

\begin{proof}
It suffices to prove the \lcnamecref{thm59}
in the case where $\ul R'=\{\ul\alpha\}$.
Then both sides equal 
\[
\sum_{\substack{t\in\ord_{\ul\alpha} x\\0<t<f(\ul\alpha)}}
\len\frak g^{\ul\alpha}_{x,t:t+}
\]
by the third part of \Cref{thm16}.
\end{proof}

To simplify the sum in \Cref{thm59}
we prove a more general result
about summation of
discretely supported functions $h:\R\to\N$.
Given a lattice $\Lambda\subset\R$,
that is, a nontrivial cyclic subgroup,
say that $h$ is \emph{$\Lambda$-periodic}
if $h(t + \lambda) = h(t)$ for all $t\in\R$ and~$\lambda\in\Lambda$.
Given a bounded interval $I\subset\R$
with endpoints $a < b$, define
\[
\sideset{}{'}\sum_I h
= \sideset{}{'}\sum_{t\in I} h(t)
= \sum_{a < t < b} h(t) + \tfrac12[a\in I]h(a)
+ \tfrac12[b\in I]h(b),
\]
where $[\cdot]$ is the Iverson bracket.
Because of the normalization at the endpoints,
this summation operator enjoys the property that
\[
\sideset{}{'}\sum_I h + \sideset{}{'}\sum_J h
= \sideset{}{'}\sum_{I\cup J} h
\]
whenever $I$, $J$, and $I\cup J$ are compact intervals.

\begin{lemma}
\label{thm17}
Let $\Lambda\subset\R$ be a free abelian group of rank one,
let $\lambda_0\defeq\min(\Lambda\cap\R_{>0})$,
let $h:\R\to\N$ be a discretely supported function,
and let $H(s) \defeq \sum_{0\leq t\leq s}' h(t)$ for $s>0$.
Suppose $h$ is even and $\Lambda$-periodic.
Then for all $s\in\tfrac12\Lambda\cap\R_{>0}$,
\[
H(s) = \frac{s}{\lambda_0}H(\lambda_0).
\]
\end{lemma}

\begin{proof}
Since $H(s + \lambda) = H(s) + H(\lambda)$,
induction reduces the proof to the case
where $s = \lambda_0$ or $s=\tfrac12\lambda_0$.
The first case is obvious;
for the second, use that $h(t) = h(\lambda_0-t)$.
\end{proof}

We can now compute a certain sum
that appears in the formal degree.

\begin{theorem}
\label{thm18}
Let $T$ be a tame maximal torus of~$G$,
let $x\in\cal A(\ul G,\ul T)$,
let $f:\underline\RS\cup\{0\}\to\R$
be a positive even function,
and let $\underline\RSS\subseteq\underline\RS$
be a subset closed under negation.
Suppose that $f(\ul\alpha)\in\tfrac12\ord(k_{\ul\alpha}^\times)$
for all $\ul\alpha\in\underline\RS$.
Then
\[
\len(\frak g_{x,0+:f}%
^{\underline\RSS})
+ \tfrac12\len(\frak g_{x,0:0+}%
^{\underline\RSS})
+ \tfrac12\len(\frak g_{x,f:f+}%
^{\underline\RSS})
= \sum_{\ul\alpha\in\underline\RSS} [k_{\ul\alpha}:k] f(\ul\alpha).
\]
\end{theorem}

\begin{proof}
By \cref{thm59}, the lefthand side
of the \lcnamecref{thm18} is
\[
\len(\frak g_{x,0+:f}%
^{\underline\RSS})
+ \tfrac12\len(\frak g_{x,0:0+}%
^{\underline\RSS})
+ \tfrac12\len(\frak g_{x,f:f+}%
^{\underline\RSS})
= \sum_{\ul\alpha\in\underline\RSS}
\sideset{}{'}\sum_{0\leq t\leq f(\ul\alpha)}
[\kappa_{\ul\alpha}:\kappa]\one_{\ord_x\ul\alpha}(t).
\]
Since $f$ is even and $\kappa_{\ul\alpha}
=\kappa_{-\ul\alpha}$,
the righthand side above is
\[
\sum_{\ul\alpha\in\underline\RSS}
\sideset{}{'}\sum_{0\leq t\leq f(\ul\alpha)}
\tfrac12[\kappa_{\ul\alpha}:\kappa]
\bigl(\one_{\ord_x\ul\alpha}(t) +
\one_{\ord_x(-\ul\alpha)}(t)\bigr).
\]
By \cref{thm16},
the function $\one_{\ord_x\ul\alpha} + \one_{\ord_x(-\ul\alpha)}$
is even and $\ord(k_{\ul\alpha}^\times)$-periodic.
Hence we may apply \cref{thm17} to conclude that
\[
\sideset{}{'}\sum_{0\leq t\leq f(\ul\alpha)}
\tfrac12[\kappa_{\ul\alpha}:\kappa]
\bigl(\one_{\ord_x\ul\alpha}(t) +
\one_{\ord_x(-\ul\alpha)}(t)\bigr)
= [k_{\ul\alpha}:k] f(\ul\alpha),
\]
using the fact that
\[
\sideset{}{'}\sum_{0\leq t\leq 1}
[\kappa_{\ul\alpha}:\kappa]\one_{\ord_x\ul\alpha}(t)
= \sum_{0\leq t<1} [\kappa_{\ul\alpha}:\kappa]\one_{\ord_x\ul\alpha}(t)
= [\kappa_{\ul\alpha}:\kappa]\ord_k(k_\alpha^\times)
= [k_{\ul\alpha}:k]. \qedhere
\]
\end{proof}

\subsection{Yu's construction}
\label{sec:aut:yu}
In this subsection we describe
Yu's supercuspidal representations,
following Hakim and Murnaghan's 
expanded exposition \cite[Section~3]{hakim_murnaghan08}
of Yu's original paper \cite{yu01}.
Yu's full construction is quite elaborate,
but fortunately, it is enough for us
to understand only the parts of the construction
needed to calculate the formal degree.

A \emph{cuspidal $\ul G$-datum} is a 5-tuple
$\Psi=(\vec{\ul G},y,\vec r,\rho,\vec\phi)$ consisting of:
\begin{itemize}
\item
a tame twisted Levi sequence $\vec{\ul G}$
such that $\ul Z^0/\ul Z$ is anisotropic,
where $\ul Z$ is the center of~$\ul G$
and $\ul Z^0$ is the center of $\ul G^0$;
\item
a point $y$ in the apartment 
of a tame maximal torus of~$\ul G^0$;
\item
an increasing sequence
$\vec r = (0 < r_0 < r_1 < \cdots < r_{d-1} \leq r_d)$
of real numbers
(if $d=0$ then we only require that $0\leq r_0$);
\item
an irreducible representation~$\rho$
of $G^0_{[y]}$ whose restriction to $G^0_{y,0+}$
is $1$-isotypic and for whom the compact induction
$\cInd_{K^0}^{G^0}\rho$ is irreducible
(hence supercuspidal); 
\item
a sequence $\vec\phi=(\phi_0,\dots,\phi_d)$ of characters, 
with $\phi_i$ a character of~$G^i$,
such that $\phi_d = 1$ if $r_d = r_{d-1}$
and otherwise $\phi_i$ has depth~$r_i$
for all~$i$.
\end{itemize}
The datum is \emph{generic} if for each~$i\neq d$
the character $\phi_i$ of~$G^i$
is $\ul G^{i+1}$-generic 
in the sense of \Cref{sec:llc:pairs},
and the datum is \emph{regular} if,
in addition, the depth-zero supercuspidal
representation $\cInd_{K^0}^{G^0}\rho$
is regular.

Many of the objects used in Yu's construction
and built from a cuspidal $\ul G$-datum
do not depend on the representations~$\rho$ and~$\vec\phi$.
To make this independence explicit,
we define a \emph{cuspidal $\ul G$-datum
without representations}
to be a 3-tuple $(\vec{\ul G},y,\vec r)$
consisting of the first three components
of a cuspidal $\ul G$-datum.

Let $\Psi=(\vec{\ul G},\vec r,y)$ be a cuspidal
$\ul G$-datum without representations.
From $\Psi$ we can construct the following subgroups:
\begin{align*}
K^0 &= G^0_{[y]}
& K^0_+ &= G^0_{y,0+} \\
K^{i+1} &= K^0\vec G^{(i+1)}_{y,(0+,s_0,\dots,s_i)}
& K^{i+1}_+ &= \vec G^{(i+1)}_{y,(0+,s_0+,\dots,s_i+)} \\
J^{i+1} &= (G^i,G^{i+1})_{y,(r_i,s_i)}
& J^{i+1}_+ &= (G^i,G^{i+1})_{y,(r_i,s_i+)} \\
K &\defeq K^d \eqdef K^{d+1}
& K_+ &\defeq K^d_+ \eqdef K^{d+1}_+.
\end{align*}
Here $0\leq i\leq d-1$ and $s_i\defeq r_i/2$.
Generally the dependence of these objects on~$\Psi$
is implicit, but if we wish to make
the dependence explicit
we indicate it with a subscript,
for instance, $K=K_\Psi$.
When $\Psi$ is regular, an additional
group can be constructed:
the maximal torus~$S$ of \Cref{sec:llc:reps},
maximally unramified in~$G^0$.
The groups $K^{i+1}$ and $K^{i+1}_+$
are particularly important;
later on, we need to express them
in the following alternative ways.
\begin{align}
\label{eq23}
K^{i+1} &= K^0\vec G^{[1,i+1]}_{y,s_0,\dots,s_i}
& K^{i+1}_+ &= K^0_+\vec G^{[1,i+1]}_{y,s_0+,\dots,s_i+} \\
\label{eq41}
K^{i+1} &= K^i J^{i+1}
& K^{i+1}_+ &= K^i_+ J^{i+1}.
\end{align}

We can now outline Yu's construction.
Let $\Psi=(\vec{\ul G},y,\vec r,\rho,\vec\phi)$
be a generic cuspidal $\ul G$-datum.
For each $0\leq i\leq d-1$ there is
a certain finite-dimensional
irreducible representation%
\footnote{Our $\rho_i$ is
Hakim and Murnaghan's $\phi_i'$.}
$\rho_i$ of~$K^{i+1}$
constructed from~$\phi_i$.
To specify $\rho_i$ precisely one uses
the theory of the Weil-Heisenberg representation,
but for our purposes, it is enough to know that
\[
\dim\rho_i = [J^{i+1}:J^{i+1}_+]^{1/2}.
\]
The representation~$\rho_i$
is then inflated to a representation%
\footnote{ Our $\tau_i$ is Hakim and Murnaghan's $\kappa_i$.}
$\tau_i$ of~$K$.
\Cref{sec:aut:reg} explains
in more detail how the inflation procedure works,
but at the moment, it is enough to know that
the inflation procedure preserves dimension.
In the edge case $i=-1$ take $\tau_{-1}$
to be the inflation (by the same procedure)
of~$\rho$ to~$K$,
and in the edge case $i=d$ take $\tau_d$
to be the restriction of~$\phi_d$ to~$K$;
we could also handle the case $i=-1$
in the same way as the case $0\leq i\leq d-1$
by defining $\rho_{-1}\defeq\rho$.
Finally, define the supercuspidal representation~$\pi$
attached to~$\Psi$ as the compact induction
\[
\pi = \cInd_K^G\tau,
\qquad
\tau\defeq\tau_{-1}\otimes\tau_0\otimes\cdots\otimes\tau_d.
\]

In summary, then, by \cref{thm20}
the formal degree of the supercuspidal representation~$\pi$
attached to a cuspidal generic Yu-datum~$\Psi$
has formal degree
\begin{equation}
\label{eq24}
\deg(\pi_\Psi,\mu)
= \frac{\dim\rho\cdot\prod_{i=0}^{d-1}
[J^{i+1}:J^{i+1}_+]^{1/2}}{\vol(K/A,\mu)}
\end{equation}
where $\ul A$ is the maximal split central subtorus of~$\ul G$
and $\mu$ is a Haar measure on~$G/A$.

It greatly simplifies the notation
in our computation of the formal degree
to reduce to the case where the center of~$\ul G$
is anisotropic, that is, $\ul A=1$.
Given a generic cuspidal $\ul G$-datum 
without representations
$\Psi=(\vec{\ul G},y,\vec r)$,
let $\overline\Psi=(\overrightarrow{\ul G^{\tn a}},y,\vec r\,)$
denote the reduction of~$\Psi$ modulo~$\ul A$:
that is, $\ul G^{\tn a,i}\defeq\ul G^{i,\tn a}
=\ul G^i\!/\!\ul A$.

\begin{lemma}
\label{thm25}
Let $\Psi$ be a generic cuspidal $\ul G$-datum
without representations
and let $\overline\Psi$
be the reduction of~$\Psi$ modulo~$\ul A$.
Then 
\begin{enumerate}
\item
the groups $K_{\overline\Psi}$
and $K_\Psi/A$
are identical as subgroups of~$G/A$,
and
\item
$[J_\Psi^{i+1}:J^{i+1}_{\Psi,+}]
= [J_{\overline\Psi}^{i+1}:J^{i+1}_{\overline\Psi,+}]$
for all $0\leq i\leq d-1$.
\end{enumerate}
\end{lemma}

\begin{proof}
For the first part, since
\[
\frac{K_\Psi}{A}
= \frac{G^0_{[y]}}{A}\cdot
\frac{A\vec G^{(d)}_{y,(0+,\vec s)}}%
{\vec G^{(d)}_{y,(0+,\vec s)}}
= \frac{G^0_{[y]}}{A}\cdot
\frac{\vec G^{(d)}_{y,(0+,\vec s)}}%
{\vec G^{(d)}_{y,(0+,\vec s)}\cap A},
\]
it suffices by \cref{thm9}
to show that
$G^0_{[y]}/\!A = (G^0\!/\!A)_{[y]}$,
and this follows immediately from
the surjectivity of the map
$\ul G\to(\ul G/\!\ul A)(k)$
and the natural identification
of the reduced buildings
of~$\ul G$ and $\ul G/\!\ul A$.

The second part follows by an argument
similar to the proof of \cref{thm9}.
\end{proof}

\subsection{Degree computation}
\label{sec:aut:comp}
In this subsection we compute the formal degree
of Yu's supercuspidal representation.
Let $\ul G$ be a reductive $k$-group
let $\Psi=(\vec{\ul G},y,\vec r,\rho,\vec\phi)$
be a cuspidal $\ul G$-datum,
let $\ul A$ be the maximal split central subtorus of~$\ul G$,
and let $\mu$ be the Haar measure on~$G\!/\!A$
attached by Gan and Gross~\cite{gan_gross99,hiraga_ichino_ikeda08erratum}
to a level-zero additive character of~$k$.

Starting from \cref{eq24},
we will reduce the problem of computing the formal degree
to the problem of computing certain subgroup indices.
We will then be in a position
to apply \cref{thm18},
finishing the calculation.
By \cref{thm25}, we may assume
for now that the center of~$G$ is anisotropic,
though of course this restriction
will have to be relaxed in the final formula.

We will start by simplifying the volume of~$K$.
To begin with,
\begin{equation}
\label{eq26}
\vol(K,\mu)^{-1}
= \vol(G_{y,0+},\mu)^{-1}
\frac{[K G_{y,0+}:K]}%
{[KG_{y,0+}:G_{y,0+}]}.
\end{equation}

\begin{lemma}
\label{thm27}
Let $G$ be a group,
let $H$ be a subgroup of~$G$,
and let $N\subseteq M$ be subgroups of~$G$
normalized by~$H$.%
\footnote{This condition is needed for $MH$
and~$NH$ to be groups; it might be possible
to weaken it.}
Then
\[
[MH : NH] = \frac{[M:N]}{[M\cap H:N\cap H]},
\]
provided that all three indices in the expression are finite.
\end{lemma}

\begin{proof}
The inclusion $M\cap H\into M$ induces an injective map
$M\cap H/N\cap H\into M/N$ which we can use to interpret
the former as a subgroup of the latter.
The group $M/N$ acts transitively by left multiplication
on the coset space $MH/NH$.
Consider the stabilizer of the identity coset under this action.
Clearly $M\cap H/N\cap H$ lies in the stabilizer,
and conversely, it is easy to see that any
representative of an element of the stabilizer
can be translated by an element of~$N$ to lie in~$M\cap H$.
So $M\cap H/N\cap H$ is the stabilizer of the identity element,
and the orbit-stabilizer theorem concludes the proof.
\end{proof}

By \Cref{eq23,thm27,thm48},
\begin{equation}
\label{eq28}
[KG_{y,0+}:K]
= [K^0G_{y,0+}:K^0\vec G^{[1,d]}_{y,(s_0,\dots,s_{d-1})}]
= \frac{[G_{y,0+}:\vec G^{[1,d]}_{y,(s_0,\dots,s_{d-1})}]}%
{[G^0_{y,0+}:G^0_{y,s_0}]}.
\end{equation}
It follows from \Cref{thm48} again that
$[KG_{y,0+}:G_{y,0+}] = [K^0:K^0_+]$.
Combining this calculation with
\cref{eq23,eq26,eq28} yields
\[
\vol(K,\mu)^{-1}
= \vol(G_{y,0+},\mu)^{-1}
\frac{[K G_{y,0+}:K]}%
{[KG_{y,0+}:G_{y,0+}]}
= \vol(G_{y,0+},\mu)^{-1}
\frac{[G_{y,0+}:\vec G^{[1,d]}_{y,(s_0,\dots,s_{d-1})}]}%
{[G^0_{[y]}:G^0_{y,s_0}]}.
\]

The volume of~$G_{y,0+}$ is known from the literature.

\begin{lemma}
Let $x\in\cal B(G)$.
Then
$\vol(G(k)_{x,0+},\mu)^{-1} = q^{(\dim G)/2} |\frak g_{x,0:0+}|^{1/2}$.
\end{lemma}

\begin{proof}
DeBacker and Reeder~\cite[Section~5.1]{debacker_reeder09}
defined a Haar measure~$\nu$ on~$G(k)$ such that
for any $x\in\cal B(G)$,
\[
\vol(G(k)_{x,0},\nu) = \frac{|G_{x,0:0+}|}{|\frak g_{x,0:0+}|^{1/2}};
\]
in particular, $\nu$ does not depend on the choice of~$x$.
Kaletha showed~\cite[Lemma~5.15]{kaletha15} that
$\nu = q^{(\dim G)/2}\mu$.
Combining these results proves the lemma.
\end{proof}

At this point, we can say that
\begin{equation}
\label{eq29}
\deg(\pi,\mu)
= \frac{\dim\rho}{[G^0_{[y]}:G^0_{y,0+}]}
q^{(\dim G)/2}|\frak g_{y,0:0+}|^{1/2}
\frac{[G_{y,0+}:\vec G^{[1,d]}_{y,(s_0,\dots,s_{d-1})}]}%
{[G^0_{[y]}:G^0_{y,s_0}]}
\prod_{i=0}^{d-1} [J^{i+1}:J^{i+1}_+]^{1/2}.
\end{equation}
We can now simplify \cref{eq29}
using our earlier results on concave functions.

\begin{lemma}
\label{thm30}
\[
|\frak g_{y,0:0+}|^{1/2}
\frac{[G_{y,0+}:\vec G^{[1,d]}_{y,(s_0,\dots,s_{d-1})}]}%
{[G^0_{y,0+}:G^0_{y,s_0}]}
\prod_{i=0}^{d-1} [J^{i+1}:J^{i+1}_+]^{1/2}
= \exp_q\biggl(
\tfrac12\len\frak g^0_{y,0:0+}
+ \tfrac12\sum_{i=0}^{d-1} r_i(|\RS_{i+1}| - |\RS_i|)
\biggr)
\]
\end{lemma}

\begin{proof}
Let $f=f_{(s_0,\dots,s_{d-1})}$
for the twisted Levi sequence~$\vec{\ul G}^{[1,d]}$.
By \cref{thm10},
\[
[J^{i+1}:J^{i+1}_+]
= |(\frak g^i,\frak g^{i+1})_{y,(r_i,s_i):(r_i,s_i+)}|
= \bigl|(\frak g^{i+1})^{\underline\RS_{i+1}
\setminus\,\underline\RS_i}_{y,s_i:s_i+}\bigr|
= \exp_q\Bigl(\sum_{\ul\alpha\in\underline\RS_{i+1}
\!\setminus\,\underline\RS_i}
\len\frak g^{\ul\alpha}_{y,f:f+}\Bigr).
\]
By \cref{thm15},
\[
[G_{y,0+}:\vec G^{[1,d]}_{y,(s_0,\dots,s_{d-1})}]
= |\vec G^{[1,d]}_{y,(0+,\dots,0+):(s_0,\dots,s_{d-1})}|
= \exp_q\Bigl(
\len \frak g^0_{y,0:s_0} +
\sum_{i=0}^{d-1}\sum_{\underline\RS_{i+1}\!\setminus\,\underline\RS_i}
\len\frak g^{\ul\alpha}_{y,0:f}
\Bigr)
\]
and
$[G^0_{y,0+}:G^0_{y,s_0}]
= \exp_q\len\frak g^0_{y,0:s_0}$,
so that
\[
\frac{[G_{y,0+}:\vec G^{[1,d]}_{y,(s_0,\dots,s_{d-1})}]}%
{[G^0_{y,0+}:G^0_{y,s_0}]}
= \exp_q\Bigl(
\sum_{i=0}^{d-1}\sum_{\underline\RS_{i+1}\!\setminus\,\underline\RS_i}
\len\frak g^{\ul\alpha}_{y,0:f}
\Bigr)
= \exp_q\Bigl(
\sum_{\ul\alpha\in\underline\RS\setminus\,\underline\RS_0}
\len\frak g^{\ul\alpha}_{y,0:f}
\Bigr).
\]
We now recognize that
\[
\frac{|\frak g_{y,0:0+}|^{1/2}}%
{|\frak g^0_{y,0:0+}|^{1/2}}
\cdot
\frac{[G_{y,0+}:\vec G^{[1,d]}_{y,(s_0,\dots,s_{d-1})}]}%
{[G^0_{y,0+}:G^0_{y,s_0}]}
\prod_{i=0}^{d-1} [J^{i+1}:J^{i+1}_+]^{1/2}
\]
equals $\exp_q$ of
\[
\sum_{\ul\alpha\in\underline\RS\setminus\,\underline\RS_0}
\sideset{}{'}\sum_{0 \leq t \leq f(\ul\alpha)}
\len\frak g^{\ul\alpha}_{y,t:t+}.
\]
By \cref{thm22}
the hypotheses of \cref{thm18}
are satisfied, and the expression above becomes
\[
\tfrac12\sum_{i=0}^{d-1} r_i(|\RS_{i+1}| - |\RS_i|)
\]
The proof is finished by recalling that
$|\frak g^0_{y,0:0+}|^{1/2}
= \exp_q\bigl(\tfrac12\len(\frak g^0_{y,0:0+})\bigr)$.
\end{proof}

We now make the reduction promised in \Cref{thm25}.
Recall the notation for $\overline\Psi$
defined immediately before \Cref{thm25}.

\setcounter{theoremx}{0}
\begin{theoremx}
\label{thm45}
Let $\ul G$ be a reductive $k$-group
and let $\Psi$ be a generic cuspidal $\ul G$-datum
with associated supercuspidal representation~$\pi$.
Then
\[
\deg(\pi,\mu)
= \frac{\dim\rho}{[G^{\tn a,0}_{[y]}:G^{\tn a,0}_{y,0+}]}
\exp_q\biggl(
\tfrac12\dim\ul G^{\tn a}
+ \tfrac12\dim\ul G^{\tn a,0}_{y,0:0+}
+ \tfrac12\sum_{i=0}^{d-1} r_i(|\RS_{i+1}| - |\RS_i|)
\biggr).
\]
\end{theoremx}

\begin{proof}
When $\ul A=1$, the formula follows from
\cref{eq29}, \cref{thm30},
and the fact that $\frak g^0_{y,0:0+}$
is the Lie algebra of~$\ul G^{\tn a,0}_{y,0:0+}$.
In general, according to \cref{eq24},
with the exception of the factor~$\dim\rho$,
the formal degree depends only on the underlying
$G$-datum without representations
$\Psi=(\vec{\ul G},y,\vec r)$.
It now suffices to observe that
by \cref{thm25}, the expression
\[
\frac{\prod_{i=0}^{d-1} [J^{i+1}:J^{i+1}_+]^{1/2}}{\vol(K/A,\mu)}
\]
is the same for both~$\Psi$ and the reduced datum
$\overline\Psi\defeq(\overrightarrow{\ul G^{\tn a}},y,\vec r\,)$.
\end{proof}

\subsection{Regular supercuspidals}
\label{sec:aut:reg}
In the special case where the supercuspidal
representation is regular,
we can further simplify the expression
for its formal degree.
We have already seen, in \Cref{sec:llc:reps},
the source of this simplification:
an arbitrary depth-zero supercuspidal~$\pi$
is compactly induced from 
a finite-dimensional representation~$\rho$
of the group $G_{[y]}$,
but when the supercuspidal~$\pi$ is regular,
we can recover a maximal torus~$S$ from~$\pi$,
and $\pi$ is induced from
a finite-dimensional representation~$\eta$
of the smaller group $SG_{y,0}$.
In fact, the former representation
is induced from the latter:
\begin{equation}
\label{eq42}
\rho = \Ind^{G_{[y]}}_{SG_{y,0}}\eta.
\end{equation}

In the depth zero case,
replacing $\rho$ by~$\eta$
in the formula for the formal degree
thus multiplies the rest of the formula
by the index $[G_{[y]}:SG_{y,0}]$.
And since $\eta$ is an extension
of a Deligne-Lusztig representation,
the literature provides a formula
for its dimension.
These observations simplify
the formal degree for depth-zero regular supercuspidals.

When the regular supercuspidal 
has positive depth, however,
there is a slight complication:
\Cref{eq42} must be propagated
from~$G_{[y]}=K^0$ to~$K$.
And to propagate the formula,
we need to understand the inflation procedure
mentioned in passing in \Cref{sec:aut:yu}.
Nonetheless, inflation is compatible with induction
in the most straightforward way,
and in the end, the effect
on the formula for the formal degree is the same.

Yu's inflation procedure is quite simple;
we explain it following
Hakim and Murnaghan \cite[Section~3.4]{hakim_murnaghan08}.
Recall \Cref{eq41},
that $K^{i+1} = K^iJ^{i+1}$.
Suppose we are given a representation~$\rho$ of~$K^i$
satisfying the following condition.
\begin{equation}
\label[condition]{eq43}
\tn{The restriction of $\rho$ to $K^i\cap J^{i+1}$ is $1$-isotypic.}
\end{equation}
Ultimately we will apply the following analysis
to the representation $\rho_i$,
which satisfies \Cref{eq43}.
Hence $\rho$ may be interpreted as (the inflation of)
a representation of the quotient group
$K^i/(K^i\cap J^{i+1})$.
Since $K^i$ normalizes~$J^{i+1}$,
the decomposition of \Cref{eq41}
becomes a semidirect product
after dividing by $K^i\cap J^{i+1}$.
We can now inflate this representation
of the quotient first to
the semidirect product $K^{i+1}/(K^i\cap J^{i+1})$,
then to the full group~$K^{i+1}$.
Let
\[
\Inf_{K^i}^{K^{i+1}}\rho
\]
denote the resulting representation.
Since this representation
is $1$-isotypic on~$J^{i+1}$,
and since $K^{i+1}\cap J^{i+2}\subseteq J^{i+1}$,
the representation satisfies
\Cref{eq43} with
$i$ replaced by $i+1$.
By induction, we can therefore define
for any $j\geq i+1$,
and in particular for $j=d$ (if $i<d$),
the inflated representation
\[
\Inf^{K^j}_{K_i}\rho
\defeq \Inf^{K^j}_{K^{j-1}}\cdots
\Inf^{K^{i+1}}_{K^i}\rho.
\]

\begin{lemma}
\label{thm44}
Let $\Psi$ be a cuspidal $\ul G$-datum;
recall the notations of \Cref{sec:aut:yu}
for the various objects attached to~$\Psi$.
Let $H^i \defeq SG_{y,0}^0K^i_+$
and let $H\defeq H^d$.
Suppose that there is
a (necessarily irreducible)
representation~$\eta$ of~$H^0$ such that
$\rho\defeq\Ind_{H^0}^{K^0}\eta$.
Then $\sigma_{-1}\defeq\Inf_{H^0}^H\eta$ is defined
and there is a canonical identification
\[
\tau = \Ind_H^K
(\sigma_{-1}\otimes\tau_0|_H\otimes\cdots\otimes\tau_d|_H)
\]
of representations of $K$.
\end{lemma}

\begin{proof}
The notation $\inf_{H^i}^{H^j}$
mimics the notation $\inf_{K^i}^{K^j}$:
the same construction works
if the symbol~$K$ is replaced everywhere by $H$.
Since $\rho$ is $1$-isotypic on~$J^1$, so is $\eta$;
hence $\sigma_{-1}\defeq\inf_{H^0}^H\eta$ is defined.
There is a canonical identification
\[
\Ind^K_H\Inf^H_{H^0}\eta = \Inf^K_{K^0}\Ind^{K^0}_{H^0}\eta,
\]
so that
$\tau = (\Ind^K_H\sigma_{-1})
\otimes\tau_0\otimes\cdots\otimes\tau_d$.
The result now follows from the 
well-known \cite[Section~III.2.11]{renard10}
formula
$(\Ind^K_H\pi_1)\otimes\pi_2
= \Ind^K_H(\pi_1\otimes(\pi_2|_H))$,
where $\pi_1$ is a representation of~$H$
and $\pi_2$ is a representation of~$K$.
\end{proof}

Kaletha showed
\cite[Section~3.4.4, proof of Lemma~3.4.20]{kaletha19a}
that for a regular Yu datum,
the finite-dimensional representation~$\rho$
has the property that
$\rho=\Ind_{SG^0_{y,0}}^{K^0}\eta$
where $\eta$ extends the inflation to $G^0_{y,0}$
of a Deligne-Lusztig induced representation
of the group $G^0_{y,0:0+}$.
We can thus compute $\dim\eta$
using Deligne-Lusztig theory.

Let's briefly recall the dimension formula
for a Deligne-Lusztig induction.
In this paragraph only,
let $\underline G$ be a reductive $\kappa$-group,
let $\underline S$ be a maximal
elliptic torus of~$\underline G$,
and let $\theta:S\to\C^\times$
be a character.
Deligne and Lusztig computed the dimension
of the virtual representation $\tn R_{(S,\theta)}$
in their original paper
\cite[Corollary~7.2]{deligne_lusztig76};
it is
\[
\dim\tn R_{(S,\theta)}
= \frac{[G:S]}{\dim\tn{St}_G}
\]
where $\tn{St}_G$ is the Steinberg
representation of~$G$.
In his classic book on representations
of finite groups of Lie type,
Carter computed
\cite[Corollary~6.4.3]{carter85}
the dimension of the Steinberg representation;
it is
\[
\log_q\dim\tn{St}_G = \tfrac12
(\dim \ul G-\dim\ul S).
\]
We can now assemble these results
to make our final formula.
Recall the notation of \Cref{thm45}.

\begin{corollary}
Let $\Psi$ be a regular generic cuspidal $\ul G$-datum
with resulting supercuspidal representation~$\pi$
and let $\ul S$ be the maximally unramified
maximal torus of~$\ul G^0$ resulting from~$\Psi$,
as explained in \Cref{sec:llc:reps}.
Then
\begin{equation}
\label{eq10}
\deg(\pi,\mu)
= 
|S^{\tn a}_{0:0+}|^{-1}
\exp_q\Bigl(
\tfrac12\dim\ul G^{\tn a}
+ \tfrac12\rank\ul G^{\tn a,0}_{y,0:0+}
+ \sum_{i=0}^{d-1} s_i(|\RS_{i+1}| - |\RS_i|)\Bigr).
\end{equation}
\end{corollary}

\begin{proof}
By \Cref{thm44}, \Cref{thm20},
and Kaletha's description of~$\rho$,
the formula of \Cref{thm45} remains true
if $\dim\rho$ is replaced by
\[
[K^0:SG^0_{y,0}]\dim\eta
= [G^{\tn a}_{[y]}:S^{\tn a}G^{\tn a,0}_{y,0}]\dim\eta.
\]
The dimension formula for~$\eta$
discussed in the paragraph above shows that
\begin{align*}
\dim\eta &= [G^0_{y,0}:S_0G^0_{y,0+}]
\exp_q\bigl(\tfrac12(\dim\ul G^0_{y,0:0+} 
- \dim\ul S_{0:0+})\bigr)^{-1} \\
&= [G^{\tn a,0}_{y,0}:
S^{\tn a}_0 G^{\tn a,0}_{y,0+}]
\exp_q\bigl(\tfrac12(\dim\ul G^{\tn a,0}_{y,0:0+} 
- \rank\ul G^{\tn a,0}_{y,0:0+})\bigr)^{-1}.
\end{align*}
The formula now follows.
\end{proof}

\section{Galois side}
\label{sec:gal}
Let $(S,\theta)$ be a tame elliptic pair.
We saw in \cref{sec:llc:params}
that when $\theta$ is extra regular,
such a pair can be extended
to a regular supercuspidal $L$-packet datum
$(S,\hat\jmath,\chi,\theta)$,
and that the resulting set of $\chi$-data
can then be used to form
the regular supercuspidal parameter
\[
\varphi_{(S,\theta)} \defeq {}^L j_\chi\circ{}^L \theta.
\]
Moreover, every regular supercuspidal parameter
arises in this way.
Our goal in this section is to compute the Galois
side of the formal degree conjecture
for the parameter~$\varphi=\varphi_{(S,\theta)}$.
As we mentioned in \Cref{sec:llc:params},
the group $S_\varphi^\natural$ is abelian
and thus has only one-dimensional
irreducible representations,
so that the Galois side of the conjecture
simplifies to
\[
\frac{|\gamma(0,\varphi,\tn{Ad},\psi)|}%
{|\pi_0(S_\varphi^\natural)|} .
\]
Moreover, the factor $|\pi_0(S_\varphi^\natural)|$
has been computed in the literature.
So our task is to compute
the absolute value
$|\gamma(0,\varphi,\tn{Ad},\psi)|$
of the adjoint $\gamma$-factor.

We start by reviewing in \Cref{sec:gal:fac}
the general definition of the $\gamma$-factor.
As a second preliminary step,
we work out in \Cref{sec:gal:chi}
how to base change the function~$r_\chi$
used to solve the extension problem
of \Cref{sec:llc:emb}.

To compute the adjoint $\gamma$-factor,
we give an explicit description of
the adjoint representation
attached to the $L$-parameter~$\varphi$.
It turns out that this representation
decomposes as a direct sum of two representations,
one coming from the maximal torus of the dual group
and the other from its root system.
We can thus compute the adjoint $\gamma$-factors separately,
in \nameCrefs{sec:gal:toral}~\ref{sec:gal:toral}
and \ref{sec:gal:root},
and multiply them together for the final answer,
in \Cref{sec:gal:summary}.
Beginning in \Cref{sec:gal:adj},
the start of the $\gamma$-factor computation proper,
we must fix~$\hat\jmath$ and~$\chi$
in the $L$-packet datum $(S,\hat\jmath,\chi,\theta)$
extending~$(S,\theta)$.

\subsection{Review of \texorpdfstring{$L$}{L}-, 
\texorpdfstring{$\varepsilon$}{\textepsilon}- and 
\texorpdfstring{$\gamma$}{\textgamma}-factors}
\label{sec:gal:fac}
The $\gamma$-factor of a representation~$(\pi,V)$ 
of the Weil group~$W_k$ is defined by the formula
\[
\gamma(s,\pi,\psi,\mu)
\defeq \varepsilon(s,\pi,\psi,\mu)
\frac{L(1-s,\pi^\vee)}{L(s,\pi)}
\]
where $\psi$ is a nontrivial additive character of~$k$
and $\mu$ is an additive Haar measure on~$k$.
Hence the $\gamma$-factor is built from two quantities,
the $L$-factor and the $\varepsilon$-factor.

In this subsection we recall the definitions 
of the $L$-factor and the $\varepsilon$-factor,
following Tate's Corvallis notes \cite{tate_corvallis2}.
Roughly speaking, the $L$-factor carries information
about the unramified part of the representation
and the $\varepsilon$-factor carries information
about the ramified part of the representation.
Since computing the absolute value of an $\varepsilon$-factor
amounts to computing an Artin conductor,
we also explain how to compute this quantity in our application,
following Chapter~VI of Serre's \emph{Local Fields}
\cite{serre79}.

The \emph{$L$-factor} of~$\pi$ is the holomorphic function
\[
L(s,\pi) \defeq \det\bigl(1 - q^{-s}\pi(\Frob) \mid V^{I_k} \bigr)^{-1}
\]
where $I_k\subset W_k$ is the inertia group
and $\Frob\in W_k$ is a Frobenius element.
Later, we use the fact that the $L$-factor is \emph{inductive}:
if $\ell\supseteq k$ is a field extension of~$k$
and $(\pi,V)$ is a finite-dimensional complex
representation of~$W_k$ then
\[
L(s,\Ind_{\ell/k}(\pi)) = L(s,\pi),
\qquad
\Ind_{\ell/k} \defeq \Ind^{W_k}_{W_\ell}.
\]

The $\varepsilon$-factor is more subtle to define
than the $L$-factor,
and most of the subtlety resides in its complex argument.
Fortunately, since we are interested
in only the absolute value of the $\gamma$-factor,
not its complex argument,
we content ourselves with a description
of the absolute value of
the $\varepsilon$-factor instead.

Changing $s$, $\psi$, or $\mu$ scales
the $\varepsilon$-factor by a known quantity.
We may thus define with no loss of information
the simplified $\varepsilon$-factor
\[
\varepsilon(\pi) \defeq \varepsilon(0,\pi,\psi,\mu)
\]
where $\psi$ has level zero,
that is, $\max\{n\mid\psi(\pi^{-n}\cal O) = 1\} = 0$,
and where the Haar measure $\mu$
is self-dual with respect to~$\psi$.
With these conventions,
the absolute value
of the $\varepsilon$-factor is
\begin{equation}
\label{eq52}
|\varepsilon(\pi)|^2 = q^{\cond\pi}
\end{equation}
where $\cond \pi$ is the Artin conductor of~$\pi$.
So computing the absolute value
of the $\varepsilon$-factor
amounts to computing the Artin conductor.

The Artin conductor is defined by the following procedure.
Given a Galois representation~$(\pi,V)$,
choose a Galois extension~$\ell$ of~$k$
such that $\pi|_{W_\ell}$ is trivial,
and let $\Gamma\!_{\ell/k}$ be the Galois group of $\ell$ over~$k$.
Then the Artin conductor satisfies the formula
\[
\cond\pi = \sum_{i\geq0}
\frac{\codim(V^{\Gamma\!_{\ell/k,i}})}%
{[\Gamma\!_{\ell/k,0}:\Gamma\!_{\ell/k,i}]},
\]
where $i\mapsto G_{\ell/k,i}$ 
is the lower numbering filtration.
This formula is independent of the choice of~$\ell$.
We can now extend the definition
to all complex representations of the Weil group,
not necessarily those of Galois type,
by stipulating that the Artin conductor
be unchanged by unramified twists.
In particular, $\cond\pi = 0$ if and only if $\pi$ is unramified,
and 
\begin{equation}
\label{eq55}
\cond\pi = \codim V^{I_k}
\end{equation}
if $\pi$ is tamely ramified.
Heuristically,
the numerical invariant $\cond\pi$ 
is an enhancement of \Cref{eq55}
that takes wild ramification into account
and measures the extent to which
$\pi$ ramifies.

When $\pi$ is irreducible,
this heuristic is made precise by the formula
\begin{equation}
\label{eq53}
\cond\pi = (\dim\pi)(1 + \depth_k\pi).
\end{equation}
In particular, \Cref{eq53} holds
if  $\pi$ is a character.
For our application, 
we need only understand
how to compute the Artin conductor
of a tamely induced representation.
Unlike the $L$-factor, the Artin conductor
is not invariant under induction.
The best we can say in general is that
given a finite extension~$\ell$ of~$k$
and a representation~$\pi$ of~$W_\ell$,
the induced representation has conductor
\begin{equation}
\label{eq54}
\cond\Ind_{\ell/k}\pi
= \ord_k(\disc_{\ell/k})\dim\pi + f_{\ell/k}\cond\pi,
\end{equation}
where $\disc_{\ell/k}$ is the discriminant
of~$\ell$ over~$k$.
But when $\ell$ is tamely ramified over~$k$
and $\pi=\chi$ is a character the formula simplifies considerably,
even if, unlike in \Cref{eq53},
the induced representation is reducible.

\begin{lemma}
\label{thm12}
Let $\ell\supseteq k$ be a finite tame extension
and let $\chi:W_\ell\to\C^\times$ be a character.
Then
\[
\cond\Ind_{\ell/k}\chi = [\ell:k](1 + \depth_k\chi).
\]
\end{lemma}

\begin{proof}
Check, using the tameness of the extension, that
$\ord_k\disc_{\ell/k} = [\ell:k] - f_{\ell/k}$.
This computation together with \Cref{eq53,eq54} yields
\[
\cond\Ind_{\ell/k}\chi
= [\ell:k] + f_{\ell/k}\depth_\ell\chi.
\]
Now use that $\depth_\ell\chi=e_{\ell/k}\depth_k\chi$.
\end{proof}

The $L$-factor, $\varepsilon$-factor, and Artin conductor
are additive in the sense that
\[
L(s,\pi) = L(s,\pi_1)L(s,\pi_2),
\qquad
\varepsilon(\pi)
= \varepsilon(\pi_1)\varepsilon(\pi_2),
\qquad
\cond\pi = \cond\pi_1 + \cond\pi_2
\]
where $\pi=\pi_1\oplus\pi_2$.
Hence the $\gamma$-factor is additive as well.
This simple but crucial fact allows us
to restrict our attention to summands
of the adjoint representation.

\subsection{Base change for \texorpdfstring{$\chi$}{\textchi}-data}
\label{sec:gal:chi}
The main goal of this subsection is to 
determine how $\chi$-data behave under base change,
that is, restriction to the Weil group
of a finite separable extension of~$k$.
Once we understand the effect of base change
for arbitrary $\chi$-data,
we study its effect on minimally ramified $\chi$-data.

Most of the definitions of this subsection
are due to Langlands and Shelstad
\cite[Section~2.5]{langlands_shelstad87},
but our treatment is also influenced
by Kaletha's recent reinterpretation
of Langlands and Shelstad's formalism \cite{kaletha19b}.

Let $\ell$ be a separable quadratic extension of~$k$.
Local class field theory shows that
the quotient $k^\times/N_{\ell/k}(\ell^\times)$
is cyclic of order two.
The \emph{quadratic sign character} 
of the extension $\ell\supseteq k$
is the character $k^\times\to\{\pm1\}$
given by projection onto this quotient.

A root $\alpha\in\RS(G,S)$
is \emph{symmetric} if it is Galois-conjugate to~$-\alpha$,
and is \emph{asymmetric} otherwise.
A symmetric root~$\alpha$ is \emph{unramified}
if the quadratic extension $k_\alpha\supset k_{\pm\alpha}$
is unramified and is \emph{ramified} otherwise.
Letting $k_{\pm\alpha}$ denote
the fixed field of the stabilizer
in~$\Gamma\!_k$ of~$\{\pm\alpha\}$,
the extension $k_{\pm\alpha}\subseteq k_\alpha$
has degree two if $\alpha$ is symmetric
and degree one if $\alpha$ is asymmetric.

\begin{definition}
\label{thm66}
Let $\RS=\RS(G,S)$.
A \emph{set of $\chi$-data} for~$(S,G)$ 
(or just $S$ if $G$ is understood) is a collection
$\chi=(\chi_\alpha:k_\alpha^\times\to\C^\times)_{\alpha\in\RS}$
of characters satisfying the following properties.
\begin{enumerate}
\item
$\chi_{-\alpha} = \chi_\alpha^{-1}$.
\item
$\chi_{\sigma\alpha}=\chi_\alpha\circ\sigma^{-1}$
for all $\sigma\in\Gamma\!_k$.
\item
If $\alpha$ is symmetric then $\chi_\alpha$
extends the quadratic sign character
of $k_\alpha\supset k_{\pm\alpha}$.
\end{enumerate}
\end{definition}

Kaletha has interpreted a set of $\chi$-data
as giving rise to a character
of a certain double cover of a torus,
and the function~$r_\chi$ as
the $L$-parameter of this character
\cite[Section~3]{kaletha19b}.
In light of that interpretation,
we would expect that restricting~$r_\chi$
to an extension of~$k$
corresponds to composing the $\chi$-data with the norm map,
in analogy with the discussion from \cref{sec:not:base}.
This turns out to be the case, as \Cref{thm32} shows.

\begin{definition}
\label{thm68}
Let $\chi$ be a set of $\chi$-data
and let $\ell$ be a finite separable extension of~$k$.
The \emph{base change} of~$\chi$ to~$\ell$ is
the $\chi$-datum~$\chi_\ell$ defined by
$\chi_{\ell,\alpha}\defeq\chi_{\alpha,\ell_\alpha/k_\alpha}$.
\end{definition}

The definition of base change makes sense only if
the formula for~$\chi_\ell$ defines a $\chi$-datum.
We should immediately check this.

\begin{lemma}
The function~$\chi_\ell$ of \Cref{thm68}
is a set of $\chi$-data.
\end{lemma}

\begin{proof}
Negation equivariance is clear.
Compatibility with the Galois group
follows from the easily verified formula
  \[
    \sigma^{-1}\circ N_{\ell_{\sigma\alpha}/k_{\sigma\alpha}}
    = N_{\ell_\alpha/k_\alpha}\circ\sigma^{-1}.
  \]
As for the third property,
if $\alpha$ is symmetric over~$\ell$
then it is also symmetric over~$k$
and the canonical map
$\Gamma_{\ell_\alpha/\ell_{\pm\alpha}}
\to\Gamma_{k_\alpha/k_{\pm\alpha}}$
is an isomorphism.
Now recall%
~\cite[(1.2.2)]{tate_corvallis2} that
the local class field theory homomorphism $W_k\to k^\times$
(whose abelianization is the Artin reciprocity isomorphism)
intertwines the inclusion $W_\ell\into W_k$
with the norm map $\ell^\times\to k^\times$.
It follows that the canonical map
  \[
    N_{\ell_\alpha/k_\alpha}(\ell_\alpha^\times)/
    N_{\ell_\alpha/k_\alpha}(\ell_{\pm\alpha}^\times)
    \to k_\alpha^\times/k_{\pm\alpha}^\times
  \]
is an isomorphism, and hence that
$\chi_{\ell,\alpha}$ extends the
quadratic sign character of 
$\ell_\alpha\supset\ell_{\pm\alpha}$.
\end{proof}

It is time to start defining the function~$r_\chi$.
The definition requires
a brief preliminary discussion
of abstract group theory.
Let $G$ be a group and~$K$ a subgroup.
A section $u:K\backslash G\to G$
of the projection $G\to K\backslash G$
-- in other words, a choice of coset representatives --
gives rise to a $K\backslash G$-indexed family
of set maps $u_x:G\to K$,
for $x\in K\backslash G$.
To define the maps~$u_x$, we write down
the element $u(x)g$ and decompose it
as a product of an element of~$K$
followed by its coset representative
in $K\backslash G$;
that is, $u_x$ is defined by the following equation:
\[
u_x(g)u(xg) = \mathop u(x)g.
\]

Before defining $r_\chi$, a word is on order
on the exact nature of the object we are defining,
since that object depends on many arbitrary choices.
A \emph{gauge} is a function $p:\RS\to\{\pm1\}$
such that $p(-\alpha) = -p(\alpha)$ for all $\alpha\in\RS$.
Our construction produces for each gauge~$p$
a cohomology class $r_{\chi,p}$
of $1$-cochains $W_k\to\widehat S$.
As the gauge~$p$ varies
the cohomology classes $r_{\chi,p}$
are not identical.
However, there is a canonical means
of relating one class to the other.
For any two gauges $p$ and~$q$,
Langlands and Shelstad constructed
a canonical $1$-cochain~$s_{q/p}$,
depending only on $\ul\RS$
and not on the choice of $\chi$-data.
By definition,
the cohomology classes~$r_{\chi,p}$,
are related by the equation
\begin{equation}
\label{eq64}
r_{\chi,q} = s_{q/p}r_{\chi,p},
\end{equation}
and the $1$-cochains $s_{q/p}$
satisfy the right compatibility conditions
to make these equations consistent
\cite[Corollary~2.4.B]{langlands_shelstad87}.
In the construction that follows
we therefore define $r_{\chi,p}$
for a particular choice of gauge,
and \Cref{eq64} then defines $r_{\chi,q}$
for every other~$q$.

We can now write down the formula defining~$r_\chi$,
making several arbitrary choices along the way.
First, choose
\begin{enumerate}
\item
a section $[\ul\alpha]\mapsto\alpha$
of the orbit map $\RS\to\ul\RS/\{\pm1\}$.
\end{enumerate}
Each $[\ul\alpha]\in\ul\RS/\{\pm1\}$
thus gives rise to
two subgroups $W_\alpha$ and $W_{\pm\alpha}$,
the stabilizers of~$\alpha$ and~$\{\pm\alpha\}$ in~$W_k$.
For each $[\ul\alpha]$, choose in addition
\begin{enumerate}
\setcounter{enumi}{1}
\item
a section $u^\alpha:W_{\pm\alpha}\backslash W_k\to W_k$, and
\item
a section $v^\alpha:W_\alpha\backslash W_{\pm\alpha}\to W_{\pm\alpha}$.
\end{enumerate}
What we call ``choosing a section''
is more commonly called ``choosing coset representatives''.
Using these choices,
define the element $r_\chi(w)$ of $\widehat S=X^*(S)_\C$ by
\begin{equation}
\label{eq37}
r_\chi(w) = \prod_{[\ul\alpha],x}
\chi_\alpha(v^\alpha_0(u^\alpha_x(w)))^{u^\alpha(x)^{-1}\alpha},
\qquad
[\ul\alpha]\in\ul R/\{\pm1\},\,x\in W_{\pm\alpha}\backslash W_k.
\end{equation}
We still have to explain the dependence on the gauge.
Use choices~(1) and~(2) above to 
define the gauge~$p:\RS\to\{\pm1\}$ by
setting $p(\beta) = 1$ if and only if
$\beta = u^\alpha(x)^{-1}\alpha$ for some
$x\in W_{\pm\alpha}\backslash W_k$.
Then \Cref{eq37} defines $r_{\chi,p}\defeq r_\chi$.
Now use \Cref{eq64} to extend the definition to all gauges.

\begin{theorem}
\label{thm32}
Let $p:\RS\to\{\pm1\}$ be a gauge,
let $\chi$ be a set of $\chi$-data for~$S$,
and let $\ell$ be a finite separable extension of~$k$.
Then
\[
r_{\chi,p}|_\ell = r_{\chi_\ell,p}
\]
for some set of auxiliary choices
((1), (2), and (3) above)
in the definition of $r_\chi$ and $r_{\chi_\ell}$.
\end{theorem}

\begin{proof}
First, some preliminary notation
on group actions.
Given a right $G$-set~$X$
and elements $x,y\in X$
that lie in the same $G$-orbit,
let $x^{-1}y$, the transporter from~$x$ to~$y$,
denote the set of elements of~$G$
taking $x$ to~$y$.
If $y=gx$ then $x^{-1}y = G_xg$
where $G_x$ is the stabilizer in~$G$ of~$x$.

Let $W\defeq W_k$ and $W'\defeq W_\ell$.
It suffices to consider the case
where $\RS$ is a transitive $(\Gamma\!_k\times\{\pm1\})$-set.
The plan of the proof is to make choices
(1), (2), and~(3) for $r_\chi$ and for $r_{\chi_\ell}$
so that the equation
$r_{\chi,p}|_\ell = r_{\chi_\ell,p}$
holds on the nose, not up to cohomology.
To get equality, not just cohomology,
some of our choices depend other choices.

Fix $\alpha\in \RS$ (choice~(1) for~$r_\chi$).
The double cosets $z\in W_{\pm\alpha}\backslash W/W'$
index the $W'$-orbits of $R/\{\pm1\}$.
Choose a section $c:W_{\pm\alpha}\backslash W/W'\to W$,
and for each double coset~$z$
let $\alpha_z\defeq c(z)^{-1}\alpha$
(choice~(1) for~$r_{\chi_\ell}$),
so that the~$\alpha_z$ form a set of
representatives for the $W'$-orbits of $\RS/\{\pm1\}$.
Choose sections $u^z:W'_{\pm\alpha}\backslash W'\to W'$
(choice~(2) for~$r_{\chi_\ell}$)
and $v:W_\alpha\backslash W_{\pm\alpha}\to W_{\pm\alpha}$
(choice~(3) for~$r_\chi$),
and use them to define sections
$v^z:W'_{\alpha_z}\backslash W'_{\pm\alpha_z}
\to W'_{\pm\alpha_z}$ (choice~(3) for $r_{\chi_\ell}$)
by the formula
\[
v^z(y) = c(z)^{-1}v(c(z)y c(z)^{-1})c(z).
\]
Define the section
$u:W_{\pm\alpha}\backslash W\to W$
(choice~(2) for~$r_\chi$)
by $u(x) \defeq c(z)u^z(y)$
where $z = xW'$ and $y = (Kc(z))^{-1}x$.

We have now made all necessary choices
to define $r_\chi$ and~$r_{\chi_\ell}$.
It remains to check that these choices
define the same gauge~$p$
and that $r_{\chi,p}|_\ell = r_{\chi_\ell,p}$.

To check that the gauges agree,
first check that
the assignment $x\mapsto(z,y)$
is a bijection from $W_{\pm\alpha}\backslash W$
to the set of pairs~$(z,y)$ with
$z\in W_{\pm\alpha}\backslash W/W'$ and
$y\in W'_{\pm\alpha_z}\backslash W'$.
In the rest of the proof, we assume
that $x$ is related to $(y,z)$ by this bijection.
Hence a root is of the form
$u(x)^{-1}\alpha$ with $x\in W_{\pm\alpha}\backslash W$
if and only if it is of the form
$u^z(y)^{-1}\alpha_z$ with
$z\in W_{\pm\alpha}\backslash W/W'$ and
$y\in W'_{\pm\alpha_z}\backslash W'$.

Recall that for each $x\in W_{\pm\alpha}\backslash W$,
there is a function $u_x:W\to W_{\pm\alpha}$ obtained from~$u$
by the equation
\[
u(x)w = u_x(w)u(xw);
\]
similarly, for each $z\in W_{\pm\alpha}\backslash W/W'$ and
and $y\in W'_{\pm\alpha_z}\backslash W'$,
there is a function $u^z_y:W'\to W'_{\pm\alpha_z}$
obtained from $u^z$ by the equation
\[
u^z(y)w' = u^z_y(w')u^z(yw').
\]
These two constructions are related
in the following way.

\begin{claim}
\label{thm34}
Let $w\in W'$.
Then $u_x(w') = c(z)u^z_y(w')c(z)^{-1}$.
\end{claim}

\begin{proof}
Let $x'=xw'$, let $z'=x'W'$,
and let $y'=(Kc(z'))^{-1}x'$,
so that $(z',y')$ is obtained from~$x'$
in the same way that $(z,y)$
was obtained from~$x$.
Expand the defining equation of $u_x(w')$:
\[
c(z)u^z(y)w' = u_x(w')c(z')u^{z'}(y').
\]
Since $z=z'$ and $y'=yw'$,
\[
c(z)^{-1}u_x(y)c(z) u^z(yw') = u^z(y)w'.
\qedhere
\]
\end{proof}

Use the sections $u$ and~$v$
to compute the $L$-parameter of~$\chi$:
\[
r_\chi(w) = \prod_x
\chi_\alpha(v_0(u_x(w)))^{u(w)^{-1}\alpha},
\qquad
x\in W_{\pm\alpha}\backslash W.
\]
Assume now that $w=w'\in W'$.
Use \Cref{thm34} to simplify the expression to
\begin{align*}
r_\chi(w') &= \prod_{z,y}
\chi_\alpha(v_0(c(z)u^z_y(w')c(z)^{-1}))^{u(w')^{-1}\alpha},
\qquad
z\in W_{\pm\alpha}\backslash W/W',\;
y\in W'_{\pm\alpha_z}\backslash W' \\
&= \prod_{z,y} \chi_\alpha(c(z)^{-1}
v^z_0(u^z_y(w'))c(z))^{u(w')^{-1}\alpha}
= \prod_{z,y} \chi_{\alpha_z}
(v^z_0(u^z_y(w')))^{u^z(w')^{-1}\alpha_z}.
\end{align*}
To complete the argument,
use the property cited above
that the local class field theory homomorphism
intertwines inclusion with the norm map.
\end{proof}

\begin{remark}
Unlike most of the other parts of this paper,
\Cref{thm32} and the definitions preceding it
do not require the torus~$S$ to split
over a tamely ramified extension;
they hold for any torus whatsoever.
\end{remark}

We can use the base change formula
to bound the ramification of~$r_\chi$.

\begin{corollary}
\label{thm35}
Let $\chi$ be a set of $\chi$-data for
a tamely ramified torus~$S$.
If each $\chi_\alpha$ is tamely ramified of finite order
then $r_\chi$ is tamely ramified of finite order.
\end{corollary}

\begin{proof}
By hypothesis, there is a finite tamely ramified
extension~$\ell$ of~$k$ such that
for each $\alpha$ the character~$\chi_{\alpha,\ell}$
is trivial, so that $\chi_\ell$ is trivial.
Then $r_\chi$ restricts trivially to~$\ell$ by \Cref{thm32}
and is therefore tamely ramified of finite order.
\end{proof}

To conclude this subsection
we define minimally ramified
$\chi$-data following Kaletha
\cite[Definition~4.6.1]{kaletha19a}.
The definition is relevant to this subsection
because a minimally ramified set of $\chi$-data
satisfies the hypotheses, and thus the conclusion,
of \Cref{thm35};
we use this observation in the proof of \Cref{thm69}.

\begin{definition}
\label{thm67}
A $\chi$-datum~$\chi$ for $S$
is \emph{minimally ramified}
if $S$ is tame and in addition
$\chi_\alpha$ is trivial for asymmetric~$\alpha$,
unramified for unramified symmetric~$\alpha$,
and tamely ramified for ramified symmetric~$\alpha$.
\end{definition}

\subsection{Adjoint representation}
\label{sec:gal:adj}
Our goal in this subsection is to describe the
adjoint representation of a regular parameter,
specifically, its decomposition
into a ``toral summand'' coming from the torus~$S$
and a ``root summand'' coming from the root system~$\RS(G,S)$.
In subsequent subsections,
we compute the $\varepsilon$-factor of both summands.

Recall from \cref{sec:llc:emb}
that the regular parameter~$\varphi_{(S,\theta)}$
is given by the formula
\[
\varphi_{(S,\theta)}(w)
= \hat\jmath\bigl(\hat\theta(w)r_\chi(w)\bigr)
n(\omega_{S,G}(w))w
\]
where $\hat\jmath:\widehat S\to\widehat G$
is an admissible embedding with image
a Galois-stable maximal torus~$\widehat T$
and $\chi$ is a certain (carefully chosen)
set of minimally ramified $\chi$-data.
We obtain the adjoint representation
from $\varphi=\varphi_{(S,\theta)}$
by composing it with
the adjoint homomorphism ${}^L G\to\GL(V)$,
where $V\defeq\hat{\frak g}/\hat{\frak z}^{\Gamma\!_k}$.
The representation decomposes as a direct sum
\[
V = V_{\tn{toral}}\oplus V_{\tn{root}}
\]
where $V_{\tn{toral}}\defeq\hat{\frak t}/\hat{\frak z}^{\Gamma\!_k}$
and where
\[
V_{\tn{root}} \defeq
\bigoplus_{\alpha\in\RS(G,S)} \hat{\frak g}_\alpha.
\]
Here $\hat{\frak g}_\alpha$ is the usual
$\alpha^\vee$-eigenline for the action
of~$\widehat S$ on~$\hat{\frak g}$,
where $\alpha^\vee$ is interpreted,
via $\hat\jmath$ and the canonical identification
$X^*(\widehat T)=X_*(T)$,
as a root of $X^*(\widehat T)$.
We call $V_{\tn{toral}}$ the \emph{toral summand}
and $V_{\tn{root}}$ the \emph{root summand}.
From our formula for~$\varphi$
we can work out the adjoint Weil actions
on these summands.

\subsection{Toral summand}
\label{sec:gal:toral}
For the toral summand,
it is useful to momentarily
consider the vector space
$\widetilde V_{\tn{toral}}\defeq\Lie(\widehat T)$
equipped with the adjoint Weil action of~$\varphi$,
so that the projection
$\widetilde V_{\tn{toral}}\to V_{\tn{toral}}$
is Weil-equivariant.
In general, for any complex torus~$T$ the natural inclusion
$X_*(T)\into\Lie(T)$ gives rise to a canonical identification
$X_*(T)_\C\simeq\Lie(T)$.
The representation $\widetilde V_{\tn{toral}}$
of the Weil group~$W_k$
is therefore the complexification
of the lattice $\Lambda = X_*(\widehat T)$,
isomorphic to $X^*(S)$ by~$X^*(\hat\jmath)$
and the canonical identification
$X_*(\widehat T)=X^*(T)$.
The Galois action on the lattice~$\Lambda$ is transferred
via this chain of identifications from
the Galois action on~$X^*(S)$ arising from
the structure of~$S$ as a torus over~$k$.
To summarize, there is an identification of representations
\[
\widetilde V_{\tn{toral}} \simeq X^*(S)_\C.
\]
Although $X^*(S^{\tn a})$ is a sublattice
of $X^*(S)$, not a quotient,
since $X^*(S^{\tn a})=X^*(S)^{\Gamma\!_k}$
the smaller lattice becomes a canonical quotient
of the larger after complexifying both.
We thus have a second identification
\[
V_{\tn{toral}} \simeq X^*(S^{\tn a})_\C.
\]
We can now compute the toral $\gamma$-factor
using the lattice
\[
M \defeq X^*(S^{\tn a})^{I_k},
\]
whose complexification is the vector space
used to compute the $L$-factor of~$V_\tn{toral}$.

\begin{lemma}
\label{thm46}
$\displaystyle|\gamma(0,V_{\tn{toral}})|
= \exp_q\bigl(\tfrac12(\dim S^{\tn a} + \dim M)\bigr)
\frac{|M_\Frob|}{|(\overline\kappa^\times\otimes M^\vee)^\Frob|}$.
\end{lemma}

\begin{proof}
We omit several details because
the calculation closely follows \cite[Section~5.4]{kaletha15}.

It is easy to dispense with the $L$-factor at $s=0$:
\[
L(0,V_{\tn{toral}})^{-1} 
= \det(1 - \Frob \mid M_\C)
= |M_\Frob|,
\]
where $M_\Frob$ denotes the coinvariants of Frobenius.
The $L$-factor at $s=1$ is
\[
L(1,V_{\tn{toral}})^{-1}
= \det(1 - q^{-1}\Frob \mid M_\C)
= (-q)^{-\dim M} \det(1 - q\Frob^{-1} \mid M_\C).
\]
The determinantal factor 
in the last equation can be rewritten as
\[
\det(1 - q\Frob^{-1} \mid M_\C)
= \det(1 - q\Frob \mid M_\C^\vee)
= \bigl|(\overline\kappa^\times\otimes M^\vee)^\Frob\bigr|,
\]
meaning that
$L(1,V_{\tn{toral}})^{-1}
= q^{-\dim M}\cdot
\bigl|(\overline\kappa^\times\otimes M^\vee)^\Frob\bigr|$.
Collecting the two $L$-factors gives
\[
\biggl|\frac{L(1,V_{\tn{toral}})}{L(0,V_{\tn{toral}})}\biggr|
= q^{\dim M}\frac{|M_\Frob|}%
{|(\overline\kappa^\times\otimes M^\vee)^\Frob|}.
\]

Since $S$ is tamely ramified,
\Cref{eq55} shows that the Artin conductor
of the toral summand is just 
\[
\cond V_\tn{toral}
= \dim(V_{\tn{toral}}/V_{\tn{toral}}^{I_k})
= \dim S^{\tn a} - \dim M.
\]
By our formula relating the Artin conductor
and the $\varepsilon$-factor, \Cref{eq52},
\[
|\varepsilon(V_{\tn{toral}})| =
\exp_q\bigl(\tfrac12(\dim S^{\tn a} - \dim M)\bigr). \qedhere
\]
\end{proof}

\subsection{Root summand}
\label{sec:gal:root}
The root summand is a direct sum
of representations induced from characters
of closed, finite-index subgroups of~$W_k$.
An element~$w\in W_k$ acts on~$V_{\tn{root}}$ 
through~$\varphi$ as follows.
First, the action of~$W_k$ on~$X^*(\widehat S) (=X_*(S))$
induces an action on the root system $\RS=\RS(G,S)$,
and the element $w$ permutes the root lines by this action.
Second, the toral element
\[
t_w \defeq \hat\jmath(\hat\theta(w)r_\chi(w))
\in \widehat T
\]
scales each root line~$\widehat{\frak g}_\alpha$
by $\alpha^\vee(t_w)$,
where $\alpha^\vee\in\RS^\vee(G,S)$ is interpreted
as a character of~$\widehat T$ using~$\hat\jmath$.
It follows that $V_{\tn{root}}$ is a direct sum
of monomial representations.
That is, for each Galois orbit
$\ul\alpha\in\underline\RS(G,S)$
the subrepresentation
\[
V_{\ul\alpha} \defeq \bigoplus_{\alpha\in\ul\alpha(\bar k)}
\widehat{\frak g}_\alpha
\]
is monomial and $V_{\tn{root}}$ is the direct sum
(over $\ul\RS(G,S)$) of these representations.
Further, after choosing a representative
$\alpha\in\ul\alpha(\bar k)$,
we can identify~$V_\alpha$ with 
the representation induced from the action
of~$W_\alpha$ on~$\hat{\frak g}_\alpha$,
a certain character~$\psi_\alpha$ of~$W_\alpha$.
The essential matter, then,
is to understand these characters~$\psi_\alpha$,
and specifically, as it turns out,
their depths.

Although the factor $n(\omega_{S,G}(w))w$
stabilizes the line~$\hat{\frak g}_\alpha$,
it may fail to centralize it:
instead, the factor scales the line by a certain sign
$d_\alpha(w)\in\{\pm1\}$.
It follows that $\psi_\alpha$ is the product
of three characters:
\[
\psi_\alpha(w)
= d_\alpha(w)
\cdot\langle\alpha^\vee,\hat\jmath(r_\chi(w))\rangle
\cdot\bigl\langle\alpha^\vee,
\bigl(\hat\jmath\circ\hat\theta\bigr)(w)\bigr\rangle
\]
where $\langle-,-\rangle$
denotes the evaluation pairing
$X^*(\widehat T)\otimes\widehat T\to\C^\times$.
There are two essential cases
in the analysis of this character,
depending on whether or not the character 
$\langle\alpha^\vee,(\hat\jmath\circ\hat\theta\bigr)
\bigr\rangle|_{W_\alpha}$
has positive depth.
By the local Langlands correspondence
the depth of this character
is the same as the depth of the character
$\theta_{k_\alpha/k}\circ\alpha^\vee:
k_\alpha^\times\to\C^\times$,
and we know something about these depths
from \Cref{sec:llc:pairs}.

\begin{lemma} \label{thm69}
The character
$\langle\alpha^\vee,\hat\jmath\circ r_\chi\rangle|_{W_\alpha}$
of~$W_\alpha$ is tamely ramified.
\end{lemma}

\begin{proof}
This is an immediate corollary of \Cref{thm35}.
\end{proof}

Since $p$ is odd (see \Cref{thm63})
the character~$d_\alpha$ is also tamely ramified,
as it takes values in~$\{\pm1\}$.
So $\langle\alpha^\vee,\bigl(\hat\jmath\circ\hat\theta\bigr)
\bigr\rangle|_{W_\alpha}$
differs from $\psi_\alpha$
by a tamely ramified character.
From this we can immediately deduce the following corollary.

\begin{corollary}
\label{thm72}
If $\depth(\theta_{k_\alpha/k}\circ\alpha^\vee)>0$
then $\depth(\theta_{k_\alpha/k}\circ\alpha^\vee)
= \depth(\psi_\alpha)$.
\end{corollary}

It remains to analyze the case where the depth of
$\theta_{k_\alpha/k}\circ\alpha^\vee$
is not positive.
We first assume that $S$ is maximally unramified in~$G$,
then remove this assumption.

\begin{lemma}
\label{thm71}
Suppose $S$ is maximally unramified.
If $\depth(\theta_{k_\alpha/k}\circ\alpha^\vee)\leq0$
then $\depth\psi_\alpha=0$.
\end{lemma}

\begin{proof}
It is clear from \Cref{thm69} that
$\depth\psi_\alpha\leq0$,
so we need only show that $\psi_\alpha$
is ramified.
Using the assumption that $\theta$ is extra regular,
Kaletha proved \cite[Proposition~5.2.7]{kaletha19a}
that the parameter $\varphi={}^Lj_\chi\circ{}^L\theta$
is \emph{regular} \cite[Definition~5.2.3]{kaletha19a},
meaning in particular that the connected centralizer
of the inertia subgroup~$I_k$ in~$\widehat G$ is abelian.
So although the full centralizer of inertia may not be abelian,
it does at least have the property
that all of its elements are semisimple.
Our proof proceeds by contradiction:
assuming that $\psi_\alpha$ is unramified,
we show that the centralizer of inertia
contains a nontrivial unipotent element
and is therefore nonabelian, a contradiction.

Since $\theta$ is regular (\Cref{thm70}),
the roots $\alpha$ with
$\depth(\theta_{k_\alpha/k}\circ\alpha^\vee)\leq0$
form a sub root system~$R_0$
of $R=R(\widehat T,\widehat G)$,
and the action of inertia on $R_0$
preserves a set~$R_0^+$ of positive roots.
Let $H \defeq \Ad(\varphi(I_k))$,
let $H\alpha$ denote the $H$-orbit of $\alpha\in R$,
and let $U_\alpha\subset\widehat G$ be
the root group for $\alpha\in R$.
Since $I_k\cap W_\alpha$ is
the inertia group of~$k_\alpha$
and $\bigoplus_{\beta\in H\alpha} \hat{\frak g}_\beta$
is a monomial representation of~$I_k$ induced from~$\psi_\alpha$,
the character~$\psi_\alpha$ is unramified
if and only if the following three groups coincide:
the stabilizer of~$U_\alpha$ in~$H$,
the centralizer of~$U_\alpha$ in~$H$,
and the centralizer of~$\alpha$ in~$H$.
Moreover, $\psi_\alpha$ is unramified
if and only if $\psi_\beta$ is unramified
for each $\beta\in H\alpha$.
Assume $\alpha\in R_0^+$ satisfies
these equivalent properties
and has maximal length among all such roots.
The proof works just as well if $\alpha\in R_0^-$,
so focus on the positive roots.

First, suppose the roots in the $H$-orbit
$H\alpha$ of~$\alpha$
are pairwise orthogonal.
Choose a nontrivial element $u_\alpha\in U_\alpha$.
For each $\beta\in H\alpha$,
choose $x\in H$ such that
$\beta=x\alpha$, and
let $u_\beta\defeq xu_\alpha$.
The element $u_\beta$ depends only
on $u_\alpha$ and~$\beta$ and not on~$x$.
Consider the product
\[
u = \prod_{\beta\in H\alpha} u_\beta.
\]
Then $u$ is invariant under~$H$,
hence centralizes inertia.
But at the same time $u$ is not semisimple
because the $H$-orbit of $\alpha$ consists
of positive roots, contradicting regularity.

In the remaining case, when the roots
in the $H$-orbit of $\alpha$ are not pairwise orthogonal,
a slight elaboration of the previous argument
yields a contradiction.
In this case the $H$-orbit of~$\alpha$
admits an involution $\beta\mapsto\bar\beta$
such that $\langle\beta,\gamma\rangle\neq 0$
(for $\beta,\gamma\in H\alpha$)
if and only if $\gamma\in\{\beta,\bar\beta\}$.
From each pair $\{\beta,\bar\beta\}$
with $\beta\in H\alpha$ choose one element,
including the element~$\alpha$,
and let $(H\alpha)_+$ be the resulting
set of orbit representatives,
so that $H\alpha=(H\alpha)_+\sqcup
\overline{(H\alpha)_+}$.
Choose a nontrivial element $u_\alpha\in U_\alpha$,
choose $x\in H$ such that $x\alpha=\bar\alpha$,
and define the element $u_{\bar\alpha}\defeq xu_\alpha$,
independent of the choice of~$x$.
The commutator subgroup
$U_{\alpha+\bar\alpha}=[U_\alpha,U_{\bar\alpha}]$
is stabilized by~$x$,
and since we assumed that $\alpha$ had maximal length
among the possible counterexamples to our theorem,
it is not centralized by~$x$.
(In fact, $x$ must act by inversion on this group
because $xu_{\bar\alpha} = u_\alpha$.)
Hence there is an element
$u_{\alpha+\bar\alpha}\in U_{\alpha+\bar\alpha}$
with
\[
u_{\alpha+\bar\alpha}^{-1}
\cdot xu_{\alpha+\bar\alpha} = [u_\alpha,u_{\bar\alpha}],
\qquad\tn{that is,}\qquad
u_{\alpha+\bar\alpha}u_\alpha u_{\bar\alpha}
= x(u_{\alpha+\bar\alpha}u_\alpha u_{\bar\alpha}).
\]
For each $\beta\in(H\alpha)_+$
choose $x\in H$ such that $\beta=x\alpha$,
let $u_\beta \defeq xu_\alpha$,
and let $u_{\bar\beta} \defeq xu_{\bar\alpha}$;
these elements are independent of the choice of~$x$.
As before, define the element
\[
u = \prod_{\beta\in(H\alpha)_+}
u_{\beta+\bar\beta}u_\beta u_{\bar\beta}.
\]
Since the action $\Ad\circ\varphi$
of wild inertia on $\widehat G^0$ is trivial,
the group $H$ acts on the factors of~$u$
through some abelian quotient.
Hence $u$ centralizes inertia but is not semisimple,
contradicting regularity.
\end{proof}

\begin{remark}
Kaletha defines an $L$-parameter to be \emph{torally wild}
if it takes wild inertia to a maximal torus of~$\widehat G$,
and shows that torally wild $L$-parameters
factor through the $L$-group of a tame maximal torus
\cite{kaletha19c}.
The proof of \Cref{thm71} shows that this 
larger class of parameters satisfies the conclusions
of the \lcnamecref{thm71}.
\end{remark}

\begin{lemma}
\label{thm31}
If $\depth(\theta_{k_\alpha/k}\circ\alpha^\vee)\leq0$
then $\depth\psi_\alpha=0$.
\end{lemma}

\begin{proof}
Recall from \Cref{sec:llc:pairs}
that there is a twisted Levi subgroup~$G^0$
of~$G$ such that $\alpha\in R(S,G^0)$ if and only if
$\depth(\theta_{k_\alpha/k}\circ\alpha^\vee)\leq0$,
and that $S$ is maximally unramified in~$G$.
\Cref{thm71} handles the case where $G=G^0$,
so we assume that $G^0\subsetneq G$.

To deal with this more general
case we factor the $L$-embedding ${}^LS\to{}^LG$
through ${}^LG^0$.
Kaletha showed \cite[Lemmas 5.2.9, 5.2.8]{kaletha19a}
that there is an $L$-embedding
${}^Lj_{G^0,G}:{}^LG^0\to{}^LG$
with the following property:
the composite parameter
$W_k\to{}^LS\to{}^LG$
is given by the formula
\[
w\mapsto
\hat\jmath\bigl(\hat\theta_b(w)\hat\theta(w)r_\chi(w)\bigr)
n(\omega_{S,G}(w))w
\]
where $\theta_b:S(k)\to\C^\times$
is tamely ramified and
$\Omega(S,G^0)^{\Gamma_k}$-invariant.
Furthermore, from the construction of
${}^Lj_{G^0,G}$ it is easy to see that
the embedding $\widehat{\frak g}^0\into\widehat{\frak g}$
is ${}^Lj_{G^0,G}$-equivariant.
In this way we reduce to the previous case of $G=G^0$
but with $\theta$ replaced by $\theta\cdot\theta_b^{-1}$.
This replacement does not affect the validity of the reduction:
since $\theta_b$ is $\Omega(S,G^0)^{\Gamma_k}$-invariant
the character $\theta'=\theta\cdot\theta_b^{-1}$
is still regular \cite[Fact 3.7.6]{kaletha19a},
and since $\theta_b$ is tamely ramified, we still have that
$\depth(\theta'_{k_\alpha/k}\circ\alpha^\vee)\leq0$.
\end{proof}

\begin{remark}
\label{thm73}
Unlike \Cref{thm72},
\Cref{thm31} does not assert
that $\psi_\alpha$ and $\theta_{k_\alpha/k}\circ\alpha^\vee$
have the same depth when the latter has depth zero.
I expect this stronger assertion to be true.
It would be enough to prove that
if $\theta$ is extra regular
then $\depth(\theta_{k_\alpha/k}\circ\alpha^\vee)\geq0$.
But I was unable to prove the stronger assertion
and a weaker statement sufficed.
\end{remark}

In summary, the root summand decomposes
as a direct sum
\[
V_{\tn{root}} = \bigoplus_{\ul\alpha\in\ul\RS(G,S)} V_{\ul\alpha},
\]
where, for any $\alpha\in\ul\alpha(\bar k)$,
the representation~$V_{\ul\alpha}$
is induced from a character~$\psi_\alpha$ of~$W_\alpha$
with known depth.
We can now easily compute the $\gamma$-factor.
Recall the Galois sets $\RS_i$
and depths~$r_i\geq0$ of \Cref{sec:llc:pairs}.

\begin{lemma}
\label{thm47}
$\displaystyle
|\gamma(0,V_{\tn{root}})|
= \exp_q\Bigl(\tfrac12|\RS|
+ \tfrac12\sum_{i=0}^{d-1} r_i(|\RS_{i+1}|-|\RS_i|)\Bigr)$.
\end{lemma}

\begin{proof}
Suppose that $\alpha\in R_{i+1}$,
for $0\leq i\leq d-1$.
\Cref{thm38,thm72,thm31} show that
$\depth_k\psi_\alpha=r_i$.
Since $L$-factors are inductive
$L(s,V_{\ul\alpha}) = L(s,\psi_\alpha)$,
and since $\psi_\alpha$ is ramified
its $L$-factor is trivial.
As for the absolute value of the $\varepsilon$-factor,
since the extension $k_\alpha\supseteq k$ is tame,
\cref{thm12} shows that
$\cond V_{\ul\alpha} = (1 + r_i)|\ul\alpha(\bar k)|$,
so that
\[
\varepsilon(V_{\ul\alpha})
= \exp_q\bigl(\tfrac12(1 + r_i)|\ul\alpha(\bar k)|\bigr).
\]
Summing over $\ul\alpha\in\ul\RS(G,S)$ finishes the proof.
\end{proof}

\subsection{Summary}
\label{sec:gal:summary}
Let $A$ be the maximal split central subtorus of~$G$,
let $G^{\tn a}\defeq G/A$,
let $S^{\tn a}\defeq S/A$,
and let $M\defeq X^*(S^{\tn a})^{I_k}$.
\Cref{thm46,thm47} show that
the absolute value of the adjoint $\gamma$-factor is
\[
|\gamma(0,V)|
= \frac{|M_{\Frob}|}{|(\overline\kappa^\times\otimes M^\vee)^\Frob|}
\exp_q\Bigl(\tfrac12\dim G^{\tn a} + \tfrac12\dim M
+ \tfrac12\sum_{i=0}^{d-1} r_i(|\RS_{i+1}|-|\RS_i|)\Bigr).
\]
Finally, since
$|\pi_0(S_\varphi^\natural)| = |X_*(\overline S)_{\Gamma\!_k}|$
\cite[Lemma~5.13]{kaletha15},
the Galois side of the formal degree conjecture is
\begin{equation}
\label{eq13}
\frac{
|M_\Frob|
}{
|X_*(S^{\tn a})_{\Gamma\!_k}|\cdot
|(\overline\kappa^\times\otimes M^\vee)^\Frob|
}
\exp_q\Bigl(\tfrac12\dim G^{\tn a} + \tfrac12\dim M
+ \tfrac12\sum_{i=0}^{d-1} r_i(|\RS_{i+1}|-|\RS_i|)\Bigr).
\end{equation}

\section{Comparison}
\label{sec:com}
In this short final section
we combine our work from
\Cref{sec:aut,sec:gal} with
several results from the literature
to show that the automorphic and Galois sides
of the formal degree conjecture are equal,
the following theorem.

\setcounter{theoremx}{1}
\begin{theoremx}
\label{thm50}
Kaletha's regular $L$-packets
satisfy the formal degree conjecture, \Cref{thm1}.
\end{theoremx}

Let $(S,\theta)$ be a tame elliptic regular pair
and let $\varphi=\varphi_{(S,\theta)}$
be the $L$-parameter attached to this pair
by the constructions of \Cref{sec:llc:params}.
The Galois side of the formal degree conjecture
for~$\varphi$ is expressed in \Cref{eq13}.
Recall the notation of \Cref{sec:gal:summary}.

The supercuspidal representations
in the $L$-packet for $\varphi$ are of the form
$\pi_{(jS,j\theta')}$ as described in
\Cref{sec:llc:params},
where $j$ ranges over conjugacy classes
of admissible embeddings $j:S\into G$.
Since $j\theta'$ and $\theta\circ j^{-1}$
differ by a tamely ramified character,
the formal degrees of $\pi_{(S,\theta)}$
and $\pi_{(jS,j\theta')}$,
as expressed in \Cref{eq10}, agree.
So on the automorphic side,
we can assume for the purpose
of computing the formal degree
that the relevant pair is $(S,\theta)$,
even though it is actually $(jS,j\theta')$.

To compute the dimension of 
the lattice $M\defeq X^*(S^{\tn a})^{I_k}$
from \Cref{sec:gal:summary},
we prove an analogue for tori
of the N\'eron-Ogg-Shafarevich criterion
for abelian varieties.

\begin{lemma}
\label{thm60}
Let $k$ be a Henselian, discretely-valued field
with residue field~$\kappa$
and let $T$ be a tame $k$-torus.
Then there is a canonical identification
$X^*(T)^{I_k} = X^*(T_{0:0+})$.
\end{lemma}

\begin{proof}
Since $X^*(T)^{I_k} = X^*(T)_{k^\tn{nr}}$
where $k^\tn{nr}$ is the maximal unramified
extension of~$k$,
we can use \'etale descent for
the Moy-Prasad filtration \cite[9.1]{yu02}
to reduce the proof to the case
where $\kappa$ is separably closed.
Let $T^{\tn s}\subseteq T$ be the maximal split subtorus,
so that $X^*(T)^{I_k} = X^*(T^{\tn s})
= X^*(T^{\tn s}_{0:0+})$
since now $I_k=\Gamma\!_k$.

It suffices to prove that
the canonical inclusion
$T^{\tn s}_{0:0+}\into T_{0:0+}$
is an isomorphism.
The proof rests on two facts from SGA~3.
Since $S_0$ is smooth and affine,
the moduli space of its maximal tori
is represented by a smooth $\cal O$-scheme 
\cite[Expos\'e~XII, Corollaire~1.10]{sga3new}.
By Hensel's Lemma
\cite[Expos\'e~XI, Corollaire~1.11]{sga3new},
every $\kappa$-point of this moduli space
lifts to a $\cal O$-point.
\end{proof}

At this point, we know that the $\exp_q$ factors
in \cref{eq10} and \cref{eq13} are equal.

\begin{lemma}[{\cite[Lemma~5.17]{kaletha15}}]
$[S^{\tn a}(k):S^{\tn a}(k)_{0+}]
= |X_*(S^{\tn a})_{I_k}^\Frob|\cdot
|(\overline\kappa^\times\otimes M^\vee)^\Frob|$.
\end{lemma}

Let's now compare the remaining factors
outside of $\exp_q$.
On the automorphic side we have
$[S^{\tn a}(k):S^{\tn a}(k)_{0+}]^{-1}$;
on the Galois side we have
\[
\frac{|M_\Frob|}%
{|X_*(S^{\tn a})_{\Gamma\!_k}|\cdot
|(\overline\kappa^\times\otimes M^\vee)^\Frob|}.
\]
The ratio of one to the other is
\[
\frac{|X_*(S^{\tn a})^{I_k}_\Frob|\cdot|X_*(S^{\tn a})^\Frob_{I_k}|}%
{|X_*(S^{\tn a})_{\Gamma\!_k}|},
\]
using that $M=X_*(S^{\tn a})^{I_k}$.
This ratio equals~$1$ \cite[Lemma~5.18]{kaletha15}.

\section*{Acknowledgments}
I would like to thank my advisor Tasho Kaletha
for proposing the problem of calculating the formal degree,
sharing his wealth of knowledge in the Langlands program,
and encouraging me when I felt hopelessly stuck.
I'm also grateful to
Atsushi Ichino for notifying me of
other cases where the formal degree conjecture
has been proved,
to Stephen DeBacker
for explaining how to compute
the apartment of an elliptic maximal torus,
and
to Karol Koziol for discussing the problem
and encouraging me to work
toward the proof of \cref{thm31}.

This research was supported
by the National Science Foundation
RTG grant DMS~1840234.

\bibliography{fdrs.bib}
\bibliographystyle{amsalpha}

\end{document}